\documentclass[reqno,oneside,11pt,letterpaper]{amsart}

\usepackage[usenames,dvipsnames]{xcolor}
\usepackage[colorlinks=true,linkcolor=Blue,citecolor=NavyBlue]{hyperref}
\usepackage{amsmath,amssymb,epsf}
\usepackage{amsfonts,amsbsy,amscd,stmaryrd}
\usepackage{latexsym,amsopn}
\usepackage{amsthm,mathtools}
\usepackage{mathrsfs}
\usepackage{esint}

\allowdisplaybreaks
 
\usepackage{graphicx}

\numberwithin{equation}{section}
\newcounter{smallarabics}
\newenvironment{arabicenumerate}
{\begin{list}{{\normalfont\textrm{(\arabic{smallarabics})}}}
  {\usecounter{smallarabics}\setlength{\itemindent}{0cm}
   \setlength{\leftmargin}{5ex}\setlength{\labelwidth}{4ex}
   \setlength{\topsep}{0.75\parsep}\setlength{\partopsep}{0ex}
   \setlength{\itemsep}{0ex}}}
{\end{list}}
\newcommand{\ben}{\begin{arabicenumerate}}
\newcommand{\een}{\end{arabicenumerate}}

\newtheorem{theorem}{Theorem}[section]
\newtheorem{assumption}{Hypothesis}[section]
\newtheorem{proposition}[theorem]{Proposition}
\newtheorem{lemma}[theorem]{Lemma}
\newtheorem{corollary}[theorem]{Corollary}
\theoremstyle{definition}
\newtheorem{definition}[theorem]{Definition}
\newtheorem{remark}[theorem]{Remark}

\newtheorem{example}[theorem]{Example}
\newcommand{\beq}{\begin{equation}}
\newcommand{\eeq}{\end{equation}}
\newcommand{\bea}{\begin{aligned}}
\newcommand{\eea}{\end{aligned}}
\newcommand{\bex}{\begin{example}}
\newcommand{\eex}{\end{example}}
\def\bel{\begin{lemma}}
\def\eel{\end{lemma}}
\def\bet{\begin{theorem}}
\def\eet{\end{theorem}}
\def\bed{\begin{definition}}
\def\eed{\end{definition}}
\def\ber{\begin{remark}}
\def\eer{\end{remark}}

\def\beproof{\noindent{\bf Proof.}\ }
\renewenvironment{proof}{\beproof}{\qed}

\renewcommand{\S}{$\mathsection$}
\renewcommand{\mod}{\,\mbox{ mod } \,}
\newcommand{\step}[1]{{\noindent\emph{Step #1.}} }
\renewcommand{\sec}[1]{\S\ref{#1}}

\def\fantom{\\ &\phantom{=}\,}
\renewcommand{\leq}{\leqslant}
\renewcommand{\geq}{\geqslant}
\newcommand{\open}[1]{\mathopen{}\mathclose{\left]#1 \right[}}
\newcommand{\clopen}[1]{\mathopen{}\mathclose{\left[#1 \right[}}
\newcommand{\opencl}[1]{\mathopen{}\mathclose{\left]#1 \right]}}
\newcommand{\norm}[1]{\left\|{#1}\right\|}
\newcommand{\module}[1]{\left|#1\right|}

\newcommand{\bra}{\langle}
\newcommand{\ket}{\rangle}
\def\st{{ \ |\  }}

\newcommand*{\defeq}{:=}				
	
\def\bdf{{boundary defining function }}
\renewcommand{\iff}{\,\xLeftrightarrow{\phantom{--}}\,}

\def\rr{{\mathbb R}}
\def\zz{{\mathbb Z}}
\def\cc{{\mathbb C}}
\def\nn{{\mathbb N}}

\def\cH{{\mathcal{H}}}
\def\cI{{\mathcal{I}}}
\def\cC{{{C}}}

\def\cD{\mathcal{D}}
\def\cX{\mathcal{X}}
\def\cJ{\mathcal{J}}
\def\cP{\mathcal{P}}

\renewcommand{\Re}{\operatorname{Re}}

\DeclareMathOperator{\Diff}{Diff}
\DeclareMathOperator{\supp}{supp}

\DeclareMathOperator{\fp}{fp}

\DeclareMathOperator{\wf}{WF}

\newcommand{\ei}{\clopen{0,\varepsilon}}
\newcommand{\diag}{\mathrm{diag}}
\newcommand{\one}{{I}}
\newcommand{\inv}{{\scriptscriptstyle (-1)}}

\newcommand{\Diffb}{{\mathrm{Diff}}_{\b}}
\newcommand{\DN}{\Lambda}
\newcommand{\wo}{\square_{h}}
\newcommand{\cons}{2^{-2\nu} \frac{\Gamma(-\nu)} {\Gamma(\nu)}}

\def\p{\partial}
\def\cf{\cC^\infty}
\def\cdf{\dot\cC^\infty}

\def\distr{\cC^{-\infty}}

\def\pX{\partial X}
\def\zero{0}
\def\12{\frac{1}{2}}

\newcommand{\dir}{{\scriptscriptstyle\rm D}}

\newcommand{\pnu}{(\nu)}
\def\b{{\rm b}}
\def\c{{\rm c}}

\def\ph{{\rm ph}}

\def\loc{{\rm loc}}

 \author{Alberto Enciso}
 \address{Instituto de Ciencias Matem\'aticas, Consejo Superior de
   Investigaciones Cient\'\i ficas, Madrid, Spain}
 \email{aenciso@icmat.es}

 \author{Gunther Uhlmann}
 \address{Department of Mathematics, University of Washington, Seattle, Washington, USA} 
 \email{gunther@math.washington.edu}
 
\author{Micha{\l} Wrochna} 
\address{Mathematical Institute, Universiteit Utrecht, Utrecht, The Netherlands \vspace{-0.3cm}} \address{Mathematics \& Data Science, Vrije Universiteit Brussel, Brussels, Belgium}
 \email{{m.wrochna@uu.nl}}

\title[The Dirichlet-to-Neumann map  on AdS spaces and holography]{The Dirichlet-to-Neumann map  on asymptotically \\ anti-de Sitter spaces and holography}

\begin{document}

\begin{abstract} We consider the Klein--Gordon equation on asymptotically anti-de Sitter spacetimes, and show that the forward Dirichlet-to-Neumann map (or  scattering matrix) is a fractional power of the boundary wave operator modulo lower order terms in the sense of paired Lagrangian distributions. We use it to show that, outside of a countable set of mass parameters, the Dirichlet-to-Neumann map determines the Taylor series of the bulk metric at the boundary,  and hence allows the recovery of a real analytic metric or Einstein metric modulo isometries. Furthermore, we prove a Lorentzian version of the Graham--Zworski theorem relating poles of the Dirichlet-to-Neumann map to conformally invariant powers of the boundary wave operator.
\end{abstract}

\maketitle

\newcommand\blfootnote[1]{%
  \begingroup
  \renewcommand\thefootnote{}\footnote{#1}%
  \addtocounter{footnote}{-1}%
  \endgroup
}



\section{Introduction}

\subsection{The Dirichlet-to-Neumann map}

If $(X,g)$ a $(d+1)$-dimensional compact Riemannian manifold with boundary~$\partial X$ and the constant~$\lambda$ is not a Dirichlet eigenvalue, the {\em Dirichlet-to-Neumann}\/ map is defined as $f\mapsto\partial_{\nu} u$, where~$u$ is the only solution to the problem
\begin{equation}\label{E.Deltag}
(\Delta_g+\lambda)u=0,\quad u|_{\partial X}=f\in C^\infty(\partial X).
\end{equation}
This map can be thought of as the square root of the Laplacian operator of the boundary~$\partial X$ up to a multiplicative constant and lower order terms. More precisely, it is an elliptic pseudodifferential operator on~$\partial X$ of order~1 whose principal symbol $\sigma_1(x,\xi)$ is the norm of~$\xi$ computed with the induced metric on $T^*\partial X$ \cite{Lee1989}. 

There is a vast literature on the Dirichlet-to-Neumann map as it plays a central role in diverse mathematical and applied contexts, ranging from water waves and spectral geometry to numerical methods and medical imaging. In particular it is featured prominently in the inverse problems literature, where one is often interested in recovering geometric information about the bulk from boundary measurements. In this setting, a fundamental property is that the full symbol of the Dirichlet-to-Neumann map  determines the Taylor series of the metric at the boundary~\cite{Lee1989}. Furthermore, the work of Lassas--Uhlmann \cite{Lassas2001} and Lassas--Taylor--Uhlmann \cite{Lassas2003} shows that two real-analytic manifolds with the same Dirichlet-to-Neumann operator are necessarily isometric. Without the analyticity assumption, Daud\'e, Helffer, Kamran and Nicoleau~\cite{DKN2024} have recently shown that there exist pairs of non-isometric metrics, of arbitrarily high (finite) $C^k$~regularity, which give rise to the same Dirichlet-to-Neumann map for $\lambda\neq0$.

Still in the Riemannian setting, suppose now that $(X,g)$ is an asymptotically hyperbolic manifold. In coordinates $(x,y^0,\dots,y^{d-1})$ near the boundary $\pX=\{x=0\}$, the metric then reads 
$
g = \big(dx^2 + h(x,y,dy)\big)/{x^2}
 $ so the metric coefficients are no longer smooth up to the boundary. One can then consider an analogue of the Dirichlet-to-Neumann map, also called {\em scattering matrix}\/ in geometric scattering theory~\cite{melrose,Guillope1997,grahamzworski}. The results in that setting are quite different: as shown by Joshi and S\'a Barreto \cite{Joshi2000}, the Dirichlet-to-Neumann map is an elliptic pseudodifferential operator whose order depends on the constant~$\lambda\in \cc$. Specifically, writing $\lambda = \nu^2-\frac{d^2}4$,   the map is, up to a multiplicative constant and lower-order terms, the $(\nu - \tfrac{d}{2})^{\mathrm{th}}$ power of the Laplacian on~$\partial X$ induced by the boundary metric $h(0)=h(0,y,dy)$. Using  earlier results of Mazzeo--Melrose \cite{guillarmou2005meromorphic} on asymptotically hyperbolic resolvents (cf.~Guillarmou \cite{guillarmou2005meromorphic}), Joshi and S\'a Barreto succeeded in showing~\cite{Joshi2000} that, again, the symbol of the Dirichlet-to-Neumann map determines the Taylor series of the asymptotically hyperbolic metric at the boundary for all $\nu$ except for a discrete set of energies, cf.~\cite{Munoz-Thon2025} for a an analogous result for a particular value of $\nu$ in that set. Furthermore, in the case of conformally compact Einstein manifolds, Guillarmou--S\'a Barreto~\cite{Guillarmou2009} showed it is possible to recover the bulk manifold up to isometries, building on earlier works of Lassas--Uhlmann \cite{Lassas2001} and Lassas--Taylor--Uhlmann \cite{Lassas2003}.
 
A deep connection with conformal geometry arises if one further assumes that the asymptotically hyperbolic manifold $(X,g)$ is Einstein (or, more generally, Einstein up to errors that vanish fast enough on~$\partial X$). In this case, Graham and Zworski showed~\cite{grahamzworski} that the Dirichlet-to-Neumann map is a meromorphic function of the mass parameter~$\nu$, and that its residues at the poles are given by the so-called ``conformally invariant powers of the Laplacian'' $L_k$, which are indeed conformally invariant differential operators on~$\partial X$, {\em natural}\/ in the sense that they can be
written in terms of covariant derivatives and curvature of a representative
metric $h(0)$, and with the same principal part as the $k^{\mathrm{th}}$ power of the Laplacian. The operators $L_k$  play an important role in conformal geometry because of their relationships with invariants, and the striking connection with analysis provided by Graham--Zworski theorem yields  a useful  tool for the  study of the latter.
 
 \medskip
 
In high energy physics, asymptotically hyperbolic spaces play the role of Riemannian analogues of asymptotically \emph{anti-de Sitter spacetimes}. The importance of these Lorentzian manifolds is highlighted by the AdS/CFT correspondence,  a conjectural duality between a quantum theory of gravity on anti-de Sitter spaces and a field theory on the boundary \cite{maldacena1999large}. The conjecture motivates the question of whether objects inherent to classical and quantum fields on the boundary determine geometric degrees of freedom in the bulk.
 
  In this paper we address this problem in the setting of linear Klein--Gordon fields on a fixed asymptotically anti-de Sitter spacetime $(X,g)$. This boils  down  to asking  whether the Lorentzian Dirichlet-to-Neumann map (called \emph{boundary-to-boundary propagator} in the physics literature) determines the metric $g$ modulo natural obstructions. 
  
  Despite analogies with the asymptotically hyperbolic setting, the Lorentzian signature is in many ways different and major difficulties arise due to the lack of ellipticity. The same issues already occur in the Lorentzian analogue of the first setting, \eqref{E.Deltag}, namely the case in which $(X,g)$ is a smooth Lorentzian manifold with time-like boundary~$\partial X$, which is more straightforward to introduce for illustration. For concreteness let us suppose that it is globally hyperbolic (in the sense of spacetimes with boundary~\cite{sanchez}) with a compact Cauchy surface. Then if $t$ is a time function~\cite{sanchez}, the forward problem for the Klein--Gordon equation
\begin{equation}\label{E.KG}
	(\square_g -\lambda)u=0,\quad u|_{\partial X}=f\in C^\infty_{\c}(\partial X), \quad  \supp u \subset \{ t \geq t_0\} \mbox{ for some } t_0\in\rr
\end{equation}
has a unique solution, and the corresponding (forward) {\em Dirichlet-to-Neumann map}\/  $f\mapsto \partial_{\nu} u$ is therefore well defined. Not much more is known however outside of situations with translation invariance or real analyticity in $t$, a notable exception being results of Alexakis--Feizmohammadi--Oksanen on recovery of zero-order terms under curvature bounds \cite{Alexakis2022,Alexakis2024}, cf.~\cite{Liu2024,Yi2025,Yi2025a}. In particular, a symbolic description  is lacking.

In the case of anti-de Sitter spaces considered here, the metric blows up at $\p X$  similarly as on asymptotically hyperbolic spaces and then Dirichlet and Neumann data are replaced by two leading asymptotics. In this setting we derive  a symbolic expansion of the Dirichlet-to-Neumann map in the calculus of \emph{paired Lagrangian distributions}  introduced by Melrose and Uhlmann~\cite{Melrose1979}.  It consists of a certain class of distributions with wave front set consisting of two intersecting Lagrangian submanifolds, which generalizes  Schwartz kernels of Fourier integral operator. One of its original applications is the symbolic construction of forward inverses away from the boundary; here instead the focus is on obtaining an expansion at the boundary and keeping track on the dependence on the metric $g$.

\subsection{Main results} 

To state our results, let us fix an asymptotically AdS space $(X,g)$ of dimension $d+1$. This means that there exists a boundary defining function~$x$ (positive in the interior of~$X$, with $\partial X=\{x=0\}$, whose differential does not vanish on~$\partial X$) such that, in coordinates $(x,y^0,\dots,y^{d-1})$ near the boundary $\pX=\{x=0\}$, the metric reads
$$
g = \frac{-dx^2 + h(x,y,dy)}{x^2}
 $$  
with $h(x)=h(x,y,dy)$ a smooth family of Lorentzian metrics on~$\partial X$. We also assume that $(X,g)$ is globally hyperbolic with a compact Cauchy surface with boundary and that the metric is even modulo $O(x^3)$; the latter property holds true for instance for Einstein metrics and has the geometric interpretation of the boundary being totally geodesic. Details are given in Hypothesis~\ref{hyp:global} in the main text.

For convenience we use the following conformally rescaled  Klein--Gordon operator
$$
P=x^{\frac{1-d}{2}-2}  (\square_g + \nu^2 - \tfrac{d^2}{4} ) x^{\frac{d-1}{2}}.
$$   
The mass parameter~$\nu$ is allowed to take any positive value, which amounts to saying that $\lambda=\tfrac{d^2}{4}-\nu^2$ takes values in $]-\infty,\frac{d^2}4[$. With this choice of rescaling, the coefficients of the second order derivatives are bounded and the decisive term in the $x$ variable is the Schr\"odinger operator with inverse square potential $- \p_x^2 + \big(\nu^2-{\textstyle\frac{1}{4}}\big){x^{-2}}$. 

With these conventions, smooth, forward supported solutions of the Klein--Gordon equation $Pu=0$ can be written as
$$
u = x^{\12-\nu} u_- + x^{\12+\nu} u_+ + \, {\rm l.o.t.}, \quad u_\pm\in C^\infty(\pX). 
$$
The two leading asymptotics $u_-$ and $u_+$ play the role of generalized Dirichlet and Neumann data, with $u_-$ a multiple of $(x^{\nu-\12}u)|_{ x=0}$ and $u_+=(x^{-2\nu} x\p_x x^{\nu-\12} u) |_{ x=0}$ provided $u$ is sufficiently regularity.     The {\em Dirichlet-to-Neumann map}\/, or scattering matrix, is then the operator $\Lambda_g(\nu)$ which maps $ u_-\mapsto u_+$ for each (approximate) forward-supported solution of the equation $Pu=0$ (see Section~\ref{ss:AdS} for details). 

Our first main result, cf.~Theorems~\ref{thm:DN} in the main text, shows that the Dirichlet-to-Neumann map~$\Lambda_g$ is in a space of paired Lagrangian distributions introduced in Subsection~\ref{ss:pld}, and that it coincides with a fractional power of the boundary wave operator $\square_{h(0)}$ modulo lower order terms: 

\begin{theorem}\label{thm:DN0} Let $(X,g)$ be a $(d+1)$-dimensional asymptotically AdS space satisfying Hypothesis~\ref{hyp:global} and let  $\nu>0$ such that $\nu\notin\frac12\nn_0$. The Dirichlet-to-Neumann map $\DN_{g}(\nu)$ of the Klein--Gordon operator $\square_g + \nu^2 - \tfrac{d^2}{4}$  is a paired Lagrangian distribution in   $I^{\nu-\12,\nu+\12}(\pX\times \pX; {\Lambda}_0,{\Lambda}_1)$  satisfying
 \beq\label{eq:toprove12}
\Lambda_{g}(\nu)= -2^{-2\nu+1} \frac{\Gamma(1-\nu)}{\Gamma(\nu)}  (\square_{h(0)})^{\nu} \mod I^{\nu-\frac{3}{2},\nu-\12}(\pX\times \pX; {\Lambda}_0,{\Lambda}_1).
 \eeq
 \end{theorem} 
 
Using the construction behind Theorem~\ref{thm:DN0} one can show that~$\Lambda_g(\nu)$ encodes much geometric information about the asymptotically AdS metric~$g$. Specifically, our second result, which is a consequence of Theorems~\ref{T.Taylor} (and Corollary \ref{T.analytic}) and \ref{E.Einstein}, can be stated as follows:

\begin{theorem}\label{thm:main1}
Consider  a pair of asymptotically AdS spacetimes $(X_1,g_1)$, $(X_2,g_2)$ satisfying the same hypotheses as above and such that $\pX_1=\pX_2=Y$. For any $\nu>0$ except possibly for a countable set, if the corresponding  Dirichlet-to-Neumann maps are equal (i.e., $\Lambda_{g_1}(\nu)=\Lambda_{g_2}(\nu)$), then the Taylor series of the metric at the boundary coincides. In particular, if the metrics are real analytic then they are isometric.

Furthermore, if $(X_1,g_1)$ and $(X_2,g_2)$ are Einstein (but not necessarily real analytic) and satisfy the null convexity condition\footnote{See Definition \ref{nullco} in the main text.},
there exists a neighborhood $X_i'$ of~$Y$ in~$X_i$ such that $(X_1',g_1)$ and $(X_2',g_2)$ are isometric. 
\end{theorem}

This provides a Lorentzian counterpart of the  results of Joshi--S\'a Barreto \cite{Joshi2000} and Guillarmou--S\'a Barreto \cite{Guillarmou2009} for asymptotically hyperbolic manifolds. The extra hypotheses in the second part of the statement enable us to use the unique continuation result of Holzegel--Shao \cite{Holzegel2023}.  We remark that the analysis of asymptotically hyperbolic Einstein manifolds by Guillarmou--S\'a Barreto \cite{Guillarmou2009} uses the real analyticity of the metric and the degenerate ellipticity of the Einstein equation at the boundary,  properties that, however, are not available in the Lorentzian case.

To conclude, our third result is a Lorentzian replacement of Graham and Zworski's theorem  relating conformal invariants at $\pX$ with poles of the Dirichlet-to-Neumann map \cite{grahamzworski}:

\begin{theorem}\label{thm:main2} Let $(X,g)$ be an Einstein asymptotically AdS spacetime as in Theorem~\ref{thm:main1}. Let $k\in \nn$ and assume $k\leq \frac{d}{2}$ if $d$ is even. Then  
 \beq\label{eq:resi}
  L_k  = (-1)^{k+1} 2^{2k} k!(k-1)!   \big( \lim_{\nu\to k} (\nu-k) \Lambda_g(\nu)\big), 
 \eeq
is a conformally invariant differential operator in $\Diff^{2k}(\pX)$ with principal part $\square_{h(0)}^k$. Furthermore, the limit in \eqref{eq:resi} equals the residue of a meromorphic family of paired Lagrangian distributions.
 \end{theorem}

Note that, strictly speaking, to make the definitions invariant one needs to make the operators act on half-densities as in e.g.~\cite{grahamzworski,Joshi2000,Guillarmou2009}; we omit this aspect in the notation throughout the paper altogether. The conformal operators $L_k$ are direct analogues of their Riemannian counterparts and play thus the same role in Lorentzian conformal geometry; in particular $L_1$ is the conformal wave operator and $L_2$ the Paneitz operator. On the other hand, the analysis of $\Lambda_g(\nu)$  in the Lorentzian setting is largely different. Also, recall that we are considering only real masses because it is not known whether the forward problem is well-posed if $\nu \not\in \rr$ outside of the product metric case, as discussed in detail in \cite{vasy2012wave} and \S{\ref{intro2}}. Nevertheless, it is possible to interpret formula \eqref{eq:resi} as the residue of a meromorphic operator-valued function by replacing $\Lambda_g(\nu)$ by the parametrix that we construct in the paper: it yields indeed a well-defined meromorphic family with the same residues.

\subsection{Method of proof, bibliographical remarks}\label{intro2} The analysis in the asymptotically hyperbolic case in  \cite{Joshi2000,grahamzworski,Guillarmou2009} relies on a precise parametrix construction of the resolvent with $0,\b$-calculus and blow-up techniques, in which ellipticity  in an appropriate degenerate sense plays a key role. In the Lorentzian case the resolvent is replaced by a forward or backward inverse with Dirichlet boundary conditions and so far no parametrix construction on AdS spaces is known which would address simultaneously the paired Lagrangian distribution behaviour in the interior and the singular structure close to the boundary. Instead, the well-posedness theory relies on energy estimates and microlocal propagation estimates due initially to Vasy \cite{vasy2012wave}, cf.~Hol\-ze\-gel \cite{holzegel2012well}, Warnick \cite{warnick2013massive}, Enciso--Kamran \cite{enciso2015singular},  Gannot--Wrochna \cite{GW}, and Dappiaggi--Marta \cite{Dappiaggi2021,Dappiaggi2022}.

Rather than trying to construct a parametrix for Schwartz kernels of inverses and follow the footsteps of \cite{Joshi2000}, we observe that the Dirichlet-to-Neumann map can be described in terms of perturbations of a normal operator. The latter is associated with the case of an exact product metric in which it is possible to use Bessel functions~\cite{Enciso2017}, concretely the Hankel transform, to handle singular aspects close to the boundary. Perturbing this does not give a good control of the forward inverse directly but we show that decaying correction terms contribute to higher order regularity terms in the resulting parametrix for the Dirichlet-to-Neumann map. 

The main tool that we develop for that purpose are families of paired Lagrangian distributions depending on a frequency parameter $\xi$ in the sense of Hankel transforms or in other words, on the spectral parameter for the  Schrödinger operator $- \p_x^2 + \big(\nu^2-{\textstyle\frac{1}{4}}\big){x^{-2}}$. 
This approach allows us to largely decouple issues at the boundary and the singular behavior inherent to hyperbolic PDE and then show eventually that $\Lambda_g(\nu)$ is a paired Lagrangian distribution which equals a multiple of the boundary wave power $(\square_{h(0)})^{\nu}$ modulo lower order terms. 

Here, we propose a definition of forward complex powers which is shown to coincide with the construction of Joshi \cite{joshi}, originally formulated on static spacetimes, and is consistent with the proposal of  Antoniano--Uhlmann \cite{Antoniano1985}. We remark that if $h(0)$ is e.g.~asymptotically Minkowski, complex powers can also be defined by the spectral theorem thanks to essential self-adjointness of $\square_{h(0)}$ \cite{vasyessential,nakamurataira,JMS} in which case they behave microlocally similarly to Feynman inverses as opposed to forward or backward inverses \cite{Dang2020,Dang2021,Dang}. We do not consider the case of Feynman powers in this paper, though it is interesting to remark that the use of self-adjointness of $\square_{h(0)}$ could lead to simplifications in a Feynman version of our problem.

Difficulties in our strategy are mostly due to the singular behaviour of integrals in $\xi$ as $x\to 0^+$, which leads us to analyze finite part integrals. Another issue is that even simple operations  involve integrating with Bessel functions $J_\nu$ and $J_{\nu+1}$ of different orders, which  leads to expressions that cannot be much simplified unless one assumes $g$ is strictly even.  They retain however some homogeneity properties that we exploit in the proofs.

The inverse problem statement in Theorem \ref{thm:main1} is then a relatively simple  corollary when combined with the already mentioned unique continuation results. It is worth mentioning that in various non-linear settings it has been shown  that, up to natural obstructions, one can recover a Lorentzian metric or related information on wave equation coefficients from the Dirichlet-to-Neumann map, among others in works by Hintz--Uhlmann--Zhai \cite{HUZ,Hintz2022a}, Kurylev--Lassas--Uhlmann \cite{KUL}, Lassas--Liimatainen--Potenciano-Machado--Leyter--Tyni \cite{L,L2} and Uhlmann--Zhai \cite{UZ20}; the techniques are not directly applicable in our case. We note that non-linearity is also exploited in recent proposals in bulk metric reconstruction from entanglement data via minimal surface area variations \cite{tzouads}. 

The proof of Theorem  \ref{thm:main2} combines arguments analogous to the Riemannian case \cite{grahamzworski}  with an analysis of the normal operator, i.e.~of the operator $P$ in the case of a product metric.  Unfortunately  in the AdS setting the presence of an imaginary part in the mass parameter $\nu$ invalidates existing energy estimates rather dramatically. We refer to Vasy \cite[\S3]{vasy2012wave} for a detailed explanation of the problem, which unfortunately persists also if one uses energy estimates in twisted Sobolev spaces (derived in the subsequent works  \cite{holzegel2012well,warnick2013massive,enciso2015singular,GW,Dappiaggi2021,Dappiaggi2022}).  In essence, the singular term $\big(\nu^2-{\textstyle\frac{1}{4}}\big){x^{-2}}$ is effectively of the same order as a second derivative so there is no natural positive definite quantity to base the estimates on if $\nu\notin \rr$.  The new idea here consists in building energy estimates on sesquilinear forms involving a certain composition of Hankel transforms that roughly speaking produces a solution with spectral parameter $\nu$ given a solution with spectral parameter $\bar \nu$, and it turns out that this yields stress-energy tensor terms that have enough positivity to play with. This complicates the estimates by adding extra commutator terms with non-local operators, but overall the modifications turn out to interact quite well with the $\b$-calculus and twisted calculus used in this setting, allowing us to show well-posedness of the forward problem for the normal operator for complex $\nu$ and conclude the proofs. 

\subsection{Structure of the paper}  The paper is organized as follows.

 Section \ref{s:preliminaries} discusses preliminaries on the Klein--Gordon equation on anti-de Sitter spaces, and in \sec{ss:sketch} gives  heuristics on the Dirichlet-to-Neumann map in the product case for the sake of illustration.

Section \ref{sec:hankel} introduces the main tool of Hankel multipliers with values in paired Lagrangian distributions and derives key results on their asymptotics at $x=0$, in particular Proposition \ref{prop:key}, which later serves to estimate contributions of various terms to the Dirichlet-to-Neumann map in paired Lagrangian classes.

Section \ref{sec:DN} applies this to give an approximation of the Dirichlet-to-Neumann map within the calculus of paired Lagrangian distribution classes (Theorem \ref{thm:DN}) and concludes in Theorem \ref{T.Taylor} that $\Lambda_g(\nu)$ determines the Taylor series of $g$ at $x=0$.

Corollaries for Einstein AdS metrics, including the inverse result in Theorem and the Graham--Zworski type result are derived in Section \ref{sec:einstein}.

Finally, Appendix \ref{sec:app} proves the well-posedness of the forward problem for the Klein--Gordon equation with complex mass parameter $\nu$ for $\Re \nu>0$ and Dirichlet boundary conditions.

\section{Preliminaries} \label{s:preliminaries}

\subsection{Asymptotically AdS spacetimes} 	\label{ss:AdS}
Let $X$ be a smooth manifold with boundary $\partial X$ of  dimension $n\geq 2$. We denote by $\cf(X)$ the space of smooth functions on $X$ (in the sense of smooth extendibility  across $\p X$, correspondingly $\cdf(X)$ the subspace of smooth functions vanishing at all orders at $\p X$, and similarly $\cf_\c(X)$ and  $\dot C^\infty_\c(X)$ their compactly supported versions.

Suppose that the interior $X^\circ$ of $X$ is equipped with a smooth Lorentzian metric $g$ of signature $(1,n-1)$. 

\begin{definition} $(X,g)$ is an \emph{asymptotically anti-de Sitter} (aAdS) spacetime if the following conditions are satisfied:
\ben
	\item If $x\in \cf(X)$ is a boundary-defining function of $\p X$, then $\hat g = x^2 g$ extends smoothly to a Lorentzian metric on $X$.\smallskip
	\item The pullback $\hat g|_{\p X}$ of $\hat g$ to the boundary has Lorentzian signature.\smallskip
	\item $\hat g^{-1}(dx,dx) = -1$ on $\partial X$.
\een
\end{definition}

By the Lorentzian analogue of an argument due to Graham--Lee \cite{graham1991einstein}, there exists a boundary-defining function $x \in \cf(X)$ such that
\begin{equation} \label{eq:specialbdf}
\hat g^{-1}(dx,dx) = -1
\end{equation}
in a neighborhood of $\p X$. Given $U \subset \p X$ with compact closure, we can always choose a collar neighborhood of $U$ diffeomorphic to $\clopen{0,\varepsilon}\times U$ in which $x \in \cf(X)$ is identified with the projection onto the first factor. With this identification near $U$,
\beq\label{eq:fform}
g = \frac{-dx^2 + h}{x^2},
\eeq
where $x\mapsto h(x)$ is a family of Lorentzian metrics on $\partial X$ depending smoothly on $x\in [0,\varepsilon)$. In particular, one can choose local coordinates $(x,y^0,\ldots,y^{n-2})$ such that $x$ is a boundary-defining function satisfying  \eqref{eq:specialbdf} and 
\[
\hat g^{-1}(dx,dy^{\alpha}) = 0 \text{ near } \p X,
\]
where $\alpha = 0,\ldots,n-2$.

In order to ensure  the  well-posedness of the forward Dirichlet and Neumann problems for the Klein--Gordon operator and ensure good behavior near $x=0$,  we make the following extra assumption.

\begin{assumption} \label{hyp:global}
We assume  that $(X,g)$ is an aAdS spacetime such that $\hat g = x^2 g$ is globally hyperbolic in the sense of spacetimes with time-like boundary \cite{sanchez}, with a compact Cauchy surface. Moreover, we assume $g$ is even modulo $O(x^3)$, i.e.~the Taylor expansion of $h$ at $x=0$ contains only even terms modulo ${O}(x^{3})$.
\end{assumption}

The assumptions imply that the boundary spacetime $(\pX,h(0))$ is globally hyperbolic in the usual boundaryless sense, and the hypotheses used in \cite{GW} for the well-posedness of forward problems are satisfied. Note that the Cauchy surface compactness is not essential for our results but is assumed for the sake of notational simplicity.

We remark that evenness of $g$ modulo ${O}(x^3)$ is implicit e.g.~in the works  \cite{holzegel2012well,warnick2013massive}, and holds true for Einstein metrics. This assumption implies in particular that $\partial X$ is totally geodesic with respect to $\hat g$. In the context  of asymptotically hyperbolic spaces, the relevance of evenness modulo ${O}(x^{2k+1})$ to the meromorphic continuation of the resolvent was observed by Guillarmou \cite{guillarmou2005meromorphic}. We note that by the same arguments as in the Riemannian signature \cite[Lem.~2.1]{guillarmou2005meromorphic}, evenness modulo ${O}(x^{2k+1})$ is a condition that does not depend on the choice of special boundary defining function.

\subsection{Klein--Gordon operator} \label{subsect:kg}

Let $d=n-1$ be the dimension of $\p X$. Although we are mostly interested in the real-mass case, it is convenient to consider the Klein--Gordon operator with complex mass 
\[
\square_g - \lambda, \mbox{ with } \lambda  \in \cc \setminus {\big[\tfrac{d^2}{4},\infty\big[\,}.
\]
Let us parametrize $\lambda = \tfrac{d^2}{4}-\nu^2$ with $\nu\in \cc$ such that $\Re \nu >0$\footnote{Note that Joshi--S\'a Barreto formulate their results in \cite{Joshi2000} in terms of the shifted variable $\zeta=\nu +\frac{d}{2}$, Graham--Zworski  use the same variable $s=\nu+\frac{d}{2}$ \cite{grahamzworski}; Enciso--Gonz\'alez--Vergara use $\alpha=\nu$.}. Instead of working with $\square_g - \lambda$ directly, one often introduces the operator
\beq\label{eq:P}
\bea
P&=  x^{\frac{1-d}{2}-2}  (\square_g - \lambda) x^{\frac{d-1}{2}}\\
 &=  x^{\frac{1-d}{2}-2}  (\square_g + \nu^2 - \tfrac{d^2}{4} ) x^{\frac{d-1}{2}}\\ &= \square_{\hat g} +   \big(\nu^2-{\textstyle\frac{1}{4}}\big){x^{-2}}.
\eea
\eeq
In the sequel we often write $P(\nu)$ to emphasize the dependence on $\nu$. More explicitly, in the coordinates described in \sec{ss:AdS} one finds that
\beq
\bea
P(\nu) &= - \p_x^2 + \big(\nu^2-{\textstyle\frac{1}{4}}\big){x^{-2}}+c(x)\big(x\p_x +\textstyle \frac{d-1}{2}\big)+\square_{h(x)}\\
&= - \p_x^2 + \big(\nu^2-{\textstyle\frac{1}{4}}\big){x^{-2}}+c(x)\big(x\p_x +\textstyle \frac{d-1}{2}\big)+ x^2 L + 
\square_{h(0)}, 
\eea
\eeq
 where $\square_{h(0)}$ is the wave operator  associated to the Lorentzian metric  $h(0)$  at the boundary,  $L\in \cf(\ei,\Diff^2(\pX))$ and $c\in \cf(X)$ is given by 
\begin{equation} \label{eq:logdetderivative}
c(x) = \12 x^{-1}\partial_x (\log |{\det h(x)}|).
\end{equation}
Note that $c$ is indeed a smooth function  because  $\det h(x)=\det h(0)+O(x^2)$ using the assumption that the metric is even modulo $O(x^3)$. An advantage of the rescaling \eqref{eq:P} is that it emphasizes the relationship with the Schrödinger operator with $x^{-2}$ potential (also called Bessel operator),
$$
N_\nu  \defeq - \p_x^2 + \big(\nu^2-{\textstyle\frac{1}{4}}\big){x^{-2}}.
$$
which has been extensively  studied in the literature  on spectral theory among others including the case $\nu\notin \rr$, see e.g.~\cite{lesch,Bruneau2011,Derezinski2017,gannot2018elliptic,Derezinski2021,Derezinski2023} and references therein.

\subsection{Twisted Sobolev spaces} In what follows we discuss well-posedness of the forward Dirichlet and Neumann problems, starting first by introducing  the relevant spaces of distributions. We follow closely \cite{GW}, with straightforward modifications that account for the fact that we work here with the rescaled operator $P$ rather than with the Klein--Gordon operator $\square_g + \nu^2 - \tfrac{d^2}{4}$, and with necessary adjustments to account for the case of $\nu\notin \rr$. 

We abbreviate 
$$
L^2(X)\defeq L^2(X, \hat g) =L^2(X,x^2 g)
$$
the $L^2$ space associated with the smoothly extensible metric $\hat g$, and denote by $L^2_\loc(X)$, resp.~$L^2_\c(X)$ its local, resp.~compactly supported analogue. The wave operator $\square_g$ is  formally self-adjoint in $L^2(X,g)$, and from the definition of $P$ \eqref{eq:P} and the relation $L^2(X)=x^{\frac{1-d}{2}} L^2(X,g)$ it follows that $P$ is formally self-adjoint in $L^2(X)$.

 To each $\nu$ we associate the following space of \emph{twisted differential operators} of order one:
\[
\Diff_{\nu}^1(X) = \{x^{-\12+\nu}A x^{\nu-\12} \st  A \in \Diff^1(X)\},
\] 
where $\Diff^m(X)$ is the space of differential operators of order $m$ with $\cf(X)$ coefficients. The space $\Diff_{\nu}^1(X)$ is independent of the choice of \bdf $x$.

\begin{definition} The local \emph{twisted Sobolev space} of order $H^1_{\nu,\loc}(X)$ is defined as follows. For $u \in \distr(X)$,
 \[
 u \in H^1_{\nu,\loc}(X) \iff Qu \in L^2_\loc(X) \text{ for all } Q \in \Diff_{\nu}^1(X).
 \]
 We also set  $H^1_{\nu,\rm c}(X) = H^1_{\nu,\loc}(X) \cap \distr_{\rm c}(X)$. \end{definition}

The spaces $H^1_{\nu,\loc}(X)$ and $H^1_{\nu,\c}(X)$ are topologized in the standard way. As usual in boundary-value problems it is useful to introduce the  corresponding  space  $\dot H^1_{\nu,\loc}(X)$, defined by taking the closure of $\dot\cC^\infty(X)$ in the $ H^1_{\nu,\loc}(X)$ topology, and its compactly supported version $\dot H^1_{\nu,\c}(X)=\dot H^1_{\nu,\loc}(X) \cap \distr_{\rm c}(X)$. For $0<\Re \nu< 1$ one can characterize $\dot{H}^1_{\nu,\loc}(X)$ as the intersection
$$
\dot{H}^1_{\nu,\loc}(X)= {H}^1_{\nu,\loc}(X)\cap x L^2_{\loc}(X), 
$$
and similarly for $\dot{H}^1_{\nu,\c}(X)$. For $\Re \nu \geq 1$ we simply have $\dot{H}^1_{\nu,\loc}(X)= {H}^1_{\nu,\loc}(X)$.

Next, we introduce Sobolev spaces referring additionally to conormal regularity, i.e.~regularity with respect to~$\mathcal{V}_\b(X)$, the space of vector fields tangent to the boundary. For the purpose of defining spaces with negative or fractional order it is convenient to use Melrose's $\Psi_\b^m(X)$ calculus. In essence, it is the natural pseudo-differential class generalizing the space of $\b$-differential operators of degree $m$, $\Diff^m_\b(X)$, generated by $k$ fold compositions of vector fields in   $\mathcal{V}_\b(X)$ for $k=0,\dots m$,  see e.g.~\cite{melrose1993atiyah,vasy2008propagation}.

\begin{definition}
	Let $k = 0,\pm 1$ and $m\geq 0$. For $u \in H^k_{\nu,\loc}(X)$,  $u \in H^{k,m}_{\nu,\b,\loc}(X)$ if and only if $Au \in H^k_{\nu,\loc}(X)$ for all $A \in \Psi_{\rm b}^m(X)$. We set  $
	H^{k,\infty}_{\nu,\b,\loc}(X) = \textstyle\bigcap_{m} H^{k,m}_{\nu,\b,\loc}(X).
	$
The spaces $\dot{H}^{k,m}_{\nu,\b,\loc}(X)$ for $k= \pm 1$ are defined analogously,  and the corresponding compactly supported spaces are defined in the usual way.
\end{definition}

\subsection{Trace operators} Assume $\Re\nu>0$.  If $u\in x^{\nu-\12} C^\infty(X) +  x^{\nu+\12} C^\infty(X)$ modulo a lower order polyhomogeneous expansion at $\p X $, then the (generalized) Dirichlet and Neumann data of $u$ are defined by extracting the (suitably normalized) two  leading asymptotics, namely
$$
 \gamma_- u = \lim_{x\to 0^+}(x^{\nu-\12}u(x)), \quad  
 \gamma_+ u =  2 \nu \fp_{x\to 0^+}(x^{-\nu-\12}u)(x)
$$ 

In practice, for the purpose analyzing boundary value problems one needs to give meaning to $\gamma_\pm$ as weighted trace maps on larger spaces.

This is relatively straightforward  for the Dirichlet weighted trace  
\beq
\gamma_- u = (x^{-\12+\nu}u)|_{\pX}
\eeq
 for $0<\Re \nu<1$, defined  initially for  $u\in x^{\12-\nu}\cf(X)$,  which then  extends  as a bounded operator
$
\gamma_- : H^1_{\nu,\loc}(X) \rightarrow H^{\nu}_{\loc}(\p X)$.
As expected, the null space  of $\gamma_-$ on $H^1_{\nu,\loc}(X)$ is precisely  $\dot{H}^1_{\nu,\loc}(X)$. Furthermore, $\gamma_-$  maps 
\beq\label{eq:gdir}
\gamma_-:H^{1,m}_{\nu,\b,\loc}(X)\to H^{\nu+m}_\loc(\pX)
\eeq 
continuously for each $m\geq 0$, as follows from the arguments in \cite[\S3.2--\S3.4]{GW} (here we allow for $\nu$ to be complex but this does not affect the proofs). If $\Re \nu \geq 1$ then $x^{\12-\nu}\cf(X)$  is not in $L^2_{\rm loc}(X)$, so in the $H^1_{\nu,\loc}(X)$ setting the extension of \eqref{eq:gdir} to $\Re \nu \geq 1$ is just $\gamma_-\equiv 0$.

\begin{remark} When $\nu \neq 1/2$, the map $\gamma_-$ depends on the choice of \bdf $x$, though only in a mild way. Namely, if $\tilde x = ax$ is another \bdf with $a \in \cf(X)$ and $a > 0$, then $
 \tilde\gamma_- u = (a|_{\p X})^{-\12+\nu}  \gamma_- u$,
where $\gamma_-$ and $\tilde\gamma_-$ are defined with respect to $x$ and $\tilde x$ respectively.
\end{remark}

The second weighted trace can  be written as
\beq\label{eq:secondtrace}
\gamma_+ u = ( x^{-2\nu+1} \p_x x^{\nu-\12} u) |_{\pX}
\eeq
for $u\in x^{\nu+\12} \cf(X)$. The definition  does not generalize well to $H^{1,m}_{\nu,\b,\loc} (X)$ due to insufficient decay  at $\pX$. However, $\gamma_+$ can be usefully extended to spaces  of approximate solutions of $P(\nu)u=0$ by means of an asymptotic expansion. Namely, 
for $m \in \rr \cup \{\pm \infty\}$ we define the spaces
\[
\cX^m_\nu = \{u \in H^{1,m}_{\nu,\b,\loc} (X) \st P(\nu)u \in H_{\nu,\b,\loc}^{0,m}(X) \},
\]
topologized by graph seminorms. The existence of an expansion for smooth  solutions of $Pu\in\cdf(X)$ is well-known; here we  use the following more precise result which is a straightforward extension of \cite[\S4.4]{GW} beyond the range $0<\nu<1$.

\begin{proposition} \label{lem:graphnormexpansion} Let $m \in \rr \cup \{\pm \infty\}$. If   $0<\Re \nu <1$  and $\nu\neq \12$, then for  $u\in \cX^m_\nu$  the restriction of $u$ to  $\{ x < \varepsilon\}$ admits an asymptotic expansion
	\begin{equation} \label{eq:graphnormexpansion}
	u =  x^{\12-\nu} u_- + x^{\12+\nu} u_+ + x^{2}H^{m+2}_{\rm b} (\ei;H^{m-3}_\loc(\p X)),
	\end{equation} 
where $u_-= \gamma_- u \in H^{m+\nu}_\loc(\p X)$ and $u_+ = (2\nu)^{-1} \gamma_+u \in H^{m-\nu}_\loc(\p X)$. Furthermore, the weighted trace map  \beq\label{eq:gpm}
 \gamma_+ :   \cX^m_\nu \to H^{m-\nu}_\loc(\p X)
 \eeq
is bounded.  If $\Re \nu \geq 1$ and $\nu\notin \frac{1}{2}\nn$ then   \eqref{eq:graphnormexpansion} holds true with $u_-=0$, Furthermore the map $\gamma_+ u = (x^{-\12-\nu}u)|_{\pX}$, defined  initially for  $u\in x^{\12+\nu}\cf(X)$, extends to a bounded operator $\gamma_+ :   H^{1,m}_{\nu,\b,\loc} (X)  \to H^{m-\nu}_\loc(\p X)$.
\end{proposition}

\subsection{Dirichlet problem and scattering matrix}

The Dirichlet and Neumann problems for $P\pnu$ are conveniently defined through sesquilinear forms, see e.g.~\cite[\S5.1]{GW}. 
We focus on the \emph{Dirichlet Klein--Gordon operator}, which maps 
$$
P_{\dir}\pnu : \dot{H}^{1,m}_{\nu,\b,\loc}(X) \rightarrow H^{-1,m}_{\nu,\b,\loc}(X)
$$
 for each $m \in \rr$. 
 
 Well-posedness of the forward Dirichlet problem was shown  by Vasy \cite{vasy2012wave}; cf.~Hol\-ze\-gel \cite{holzegel2012well} for a related study of the Cauchy problem. The Robin problem  was studied by Warnick \cite{warnick2013massive} in the case of metrics that are even modulo $O(x^3)$, and a stronger version was obtained by Gannot--Wrochna in \cite{GW}; see also Enciso--Kamran \cite{enciso2015singular} for regularity in higher-order twisted Sobolev spaces and for the non-linear case, and Dappiaggi--Marta \cite{Dappiaggi2021,Dappiaggi2022} for more general boundary conditions. Here we use the formalism from \cite[\S7.1]{GW} to also derive result for complex $\nu$  in the case of product metrics and Dirichlet boundary conditions.

\begin{theorem}\label{thm:Pinv} Let either $\nu>0$, or $\Re \nu>0$ in the case $\hat g$ is a product metric (that is, $h(x)=h(0)$ is independent of~$x$). Assume that Hypothesis \ref{hyp:global} holds true. Then there exists a unique forward Dirichlet propagator $P_{{\dir},+}^{-1}\pnu$, i.e.~a continuous operator
\beq\label{pp0}
P_{\dir,+}^{-1} \pnu: H^{-1,m+1}_{\nu,\b,\c}(X)\to \dot{H}^{1,m}_{\nu,\b,\loc}(X)
\eeq
such that for all $v\in H^{-1,m+1}_{\nu,\b,\c}(X)$, $u=P_{\dir,+}^{-1}\pnu v$ solves the forward Dirichlet problem:
$$
 \gamma_- u=0, \quad P \pnu u=v,    \quad  \supp u \subset \{ t \geq t_0\} \mbox{ for some } t_0\in\rr.  
$$
\end{theorem}

The proof in the case  $\nu\notin \rr$ uses Hankel transforms introduced in the next chapter and is postponed to Appendix \ref{sec:app}. For the real case, see~\cite[Thm.~4.16]{vasy2012wave}.

\medskip

As a consequence of Theorem \ref{thm:Pinv} one  gets   as in  \cite[Thm.~8.11]{vasy2012wave}   the well-posedness of the forward problem with prescribed Dirichlet data.  Namely, given  $v\in \dot{H}^{-1,m+1}_{\nu,\b,\c}(X)$ and $f\in H^{\nu+m}_{\rm c}(\pX)$  there exists a unique solution $u\in {H}^{1,m}_{\nu,\b,\loc}(X)$ of the forward problem
\beq\label{eq:forward}
\gamma_-u=f, \quad P\pnu u=v,    \quad  \supp u \subset \{ t \geq t_0\} \mbox{ for some } t_0\in\rr.  
\eeq

The Dirichlet backward propagator $P_{{\dir},-}^{-1}\pnu$ is defined similarly with support propagating to the past.  By uniqueness, the formal adjoint of $P_{{\dir},\pm}^{-1}(\nu)$ equals $P_{{\dir},\mp}^{-1}(\bar \nu)$. 

In what follows, for various spaces of distributions such as ${H}^{1,s}_{\nu,\b}(X)$ we denote by ${H}^{1,s}_{\nu,\b,\pm}(X)$ the subspace of future, resp.~past supported elements, equipped with the Fréchet topology given by appropriately localized seminorms. This type of spaces are particularly useful when composing forward inverses; in fact, the proof of Theorem \ref{thm:Pinv} gives the stronger statement
$$
P_{\dir,\pm}^{-1} \pnu: H^{-1,m+1}_{\nu,\b,\pm}(X)\to \dot{H}^{1,m}_{\nu,\b,\pm}(X).
$$

By repeating the arguments of Graham--Zworski \cite[Prop.~3.5]{grahamzworski} we obtain the existence of a holomorphic family of \emph{forward Poisson operators} characterized as follows.

\begin{proposition}\label{prop:gz1} There exists a unique real-analytic family of operators
$$
\cP_+(\nu): C^\infty_+(\p X)\to C^\infty_+(X^\circ), \quad \nu>0,
$$
such that $P(\nu)\cP_+(\nu)=0$ and 
\begin{align*}
\cP_+(\nu) f =\begin{cases} x^{\nu-\12} F + x^{\nu+\12} G &  \mbox{ if } \nu\notin \12\nn, \\ 
x^{-\frac{l}{2}} F +  x^{\frac{l}{2}}\log x  \,G &    \mbox{ if }  \nu= \frac{l}{2}, \ l \in \nn, 
\end{cases}
\end{align*}
where $F,G\in C^\infty_+(X)$ is such that $F|_{\pX}= f$.
\end{proposition}
\begin{proof} The arguments in \cite[\S3]{grahamzworski}  apply almost verbatim in the Lorentzian case with the difference that the resolvent is replaced by $P_{\dir,+}^{-1} \pnu$. The extra input that we real analyticity i of $\nu \mapsto P_{\dir,+}^{-1}$ for $\nu>0$, with values in, say, bounded maps $L^2_+(X)\to L^2_+(X)$; this is however a straightforward consequence of $P_{\dir,+}^{-1}\pnu: L^2_+(X)\to L^2_+(X)$ being the left inverse of $P_{\dir}(\nu): \dot H_+^{0}(X)\to L^2_+(X)$.  
\end{proof}

\begin{definition} For $\Re \nu >0$, $\nu\notin \12\nn$, the (forward) \emph{scattering matrix}, or \emph{Dirichlet-to-Neumann map}, is by definition the map
\beq
\DN_{g}(\nu)  : C_{+}^\infty(\pX)  \ni f \mapsto (2\nu)^{-1}\gamma_+ u\in C^\infty_+(\p X),
\eeq
where given $f$, $u$ is the unique solution of the forward problem \eqref{eq:forward} with $v=0$. 
\end{definition}

\begin{proposition}\label{cor:DN} For $\Re \nu >0$, $\nu\notin \12\nn$, 
if well-defined, the forward scattering matrix $\DN_{g}(\nu)  : \gamma_- u \mapsto \gamma_+ u$ equals
 \beq
 \DN_{g}(\nu) =  (2\nu)^{-1} \gamma_+ P_{\dir,+}^{-1}\pnu\gamma_+^*,
 \eeq
 where $P_{\dir,+}^{-1} \pnu\gamma_+^*$ is defined as the adjoint of $\gamma_+ P_{\dir,-}^{-1}(\bar\nu)$.
\end{proposition}
\begin{proof} By \eqref{eq:gpm} and mapping properties of  $P_{\dir,-}^{-1}$, 
$$
\gamma_+ P_{\dir,-}^{-1}(\bar \nu) :  H^{0,-m}_{\nu,\b,-}(X) \to H_-^{-m-1-\nu}(\pX) 
$$
is well-defined and continuous. Thus the formal adjoint maps continuously
$$
P_{\dir,+}^{-1}(\nu) \gamma_+^*:  H_+^{\nu+m+1}(\pX)   \to \dot H^{0,m}_{\nu,\b,+}(X).
$$
For all $u_1\in \cX^m$  and $u_2\in H_{\nu,\rm c}^{-m}(X)$ by \cite[\S4.5]{GW} we have the Green's formula
\beq\label{eq:green}
(P(\bar{\nu}) u_1 | u_2)_{L^2(M)}- (u_1|P(\nu) u_2)_{L^2(M)} = (\gamma_+ u_1 | \gamma_- u_2)_{L^2(\pX)}-  (\gamma_- u_1 | \gamma_+ u_2)_{L^2(\pX)},
\eeq
where $(\cdot|\cdot)_{L^2(\p X)}$ is the pairing on the boundary defined using the volume density of $h(0)$. We note that the condition of compact support can be dropped if $u_1$ is past supported and $u_2$ is future supported.  By applying this to $u_1=P_{\dir,-}^{-1}(\bar\nu) \gamma_+^*f$,  and $u_2=\cP(\nu)f_2$, since $\gamma_-u_1=0$ and $\gamma_- u_2 = f_2$ we obtain   
$$
( \gamma_+^*f | u_2)_{L^2(M)}=  (\gamma_+ P_{\dir,-}^{-1}(\bar\nu)\gamma_+^*f | \gamma_- u_2)_{L^2(\pX)},
$$
hence 
$$
( f |  \Lambda_g(\nu)\gamma_ - u_2)_{L^2(M)}=  (2\nu)^{-1}(\gamma_+ P_{\dir,-}^{-1}(\bar \nu) \gamma_+^*f | \gamma_- u_2)_{L^2(\pX)}.
$$
Since $f$ and  $\gamma_- u_2 = f_2$ was arbitrary we conclude that $\Lambda_g(\nu)^*= (2\bar \nu)^{-1}\gamma_+ P_{\dir,-}^{-1}(\bar \nu) \gamma_+^*$   hence $\Lambda_g(\nu)=(2\nu)^{-1}\gamma_+ P_{\dir,+}^{-1}(\nu)\gamma_+^*$ as claimed. 
\end{proof}


\medskip

As a corollary we obtain that 
 \beq\label{eq:defu}
\cP_+(\nu) f= P_{\dir,+}^{-1}(\nu)\gamma_+^*   f, \quad  f\in C_{\rm c}^\infty(\p X),
\eeq
where $P_{\dir,+}^{-1}(\nu)\gamma_+^* $ is defined by duality as before.
In fact, the right hand side  defines a forward solution $u$ with Neumann data $\gamma_+ u= \gamma_+ P_{\dir,+}^{-1}(\nu)\gamma_+^*   f=\Lambda_g(\nu) f$, hence with Dirichlet data $\gamma_- u=f$.

\medskip

From now on we  simply write  $P^{-1}$ instead of $P_{\dir,+}^{-1}(\nu)$. Note that we consider  forward problems for the sake of definiteness,  but one could equally well consider backward problems.

\subsection{Dirichlet-to-Neumann map in product case} \label{ss:sketch} \label{ss:productcase} For the sake of illustration let us explain how the Dirichlet-to-Neumann map for $\Re \nu >0$, $\nu\notin \frac{1}{2}\nn$,  can be formally obtained in the \emph{product case}, by which we mean that $h(x)=h(0)$ is $x$-independent (in that case the assumption $\nu\notin \nn$  suffices; this  boils down to the evenness of $h(x)$). In this special case the operator $P=P(\nu)$, denoted henceforth by $P_0$ to distinguish it from $P$ in the general case, is given by
$$
P_0=N_\nu+a =   - \p_x^2 + \big(\nu^2-{\textstyle\frac{1}{4}}\big){x^{-2}}+a,
$$
where $a=\square_h$.  As our discussion is for the moment purely heuristic, we treat $a$ as a real number and then the problem amounts to ODE analysis.

Let $J_\nu(x)$ be the Bessel function of order $\nu$. Let us introduce the Hankel and inverse  Hankel transforms:
$$
\big(\mathcal{H}_\nu f\big)(\xi):=  \int_0^\infty  (\xi x)^{\12} J_\nu(\xi x) f(x)  dx, \quad
\big(\mathcal{H}_\nu^{-1} g\big)(x):=\int_0^\infty  (x \xi)^{\12} J_\nu(x\xi) g(\xi)   d\xi.
$$
Then ${N}_\nu$ is diagonalized by the Hankel transform in the sense that 
\beq
{N}_\nu =  \mathcal{H}_\nu^{-1} \circ  \xi^2 \circ  \mathcal{H}_\nu,
\eeq
where $\xi^2$ is understood as a multiplication operator. Recall that the Bessel functions have asymptotics at zero:
\beq\label{eq3}
J_\nu(x)= \frac{2^{-\nu}}{\Gamma(1+\nu)} x^{\nu} ( 1+O(x^2)).
\eeq
Formally at least, that the Dirichlet inverse of $P_0$ is $
 P^{-1}_{0} =   \mathcal{H}_\nu^{-1} (\xi^2+a)^{-1}  \mathcal{H}_\nu $,
and its Schwartz kernel is
\beq\label{eq2}
 P^{-1}_{0}(x,x')= \int_0^\infty  x^{\12} J_\nu(x\xi )   (\xi^2+a)^{-1}  J_\nu(\xi x') (x')^{\12}   \xi d\xi.
\eeq
Let us also recall that the Dirichlet and Neumann maps are given respectively by
$$
 \gamma_- u = \lim_{x\to 0^+}(x^{\nu-\12}u(x)), \quad  
 \gamma_+ u =  2 \nu \fp_{x\to 0^+}(x^{-\nu-\12}u)(x),
$$
and in the special case of vanishing Dirichlet data $\gamma_- u =0$  we simply have $\gamma_+ u=2 \nu (x^{-\nu-\12}u)(x)|_{x=0}$. We can now formally compute the Dirichlet-to-Neumann map using \eqref{eq2} and the asymptotics \eqref{eq3} at zero (we will later show that this is the only regime that matters and that the finite part makes sense):
\beq\label{eq:key}
\bea
\DN_{g}(\nu)  &=  (2\nu)^{-1}\gamma_+  P_{0}^{-1}  \gamma_+^* = \fp_{x,x'\to 0} \big(  x^{-2\nu}  P_0^{-1}(x,x') (x')^{-2\nu}\big)|_{x=x'=0} \\
& =  \frac{2^{-2\nu+1}\nu}{\big(\Gamma(1+\nu)\big)^2} \fint_0^\infty  \xi^{2\nu+1}(\xi^2 +a)^{-1} d\xi  \\
 &  = \frac{2^{-2\nu} }{\Gamma(\nu)\Gamma(1+\nu)} \fint_0^\infty  \lambda^{\nu}(\lambda + a)^{-1}  d \lambda  = \cons   a^{\nu},
\eea 
\eeq
at least away from the poles of $\Gamma(1-\nu)$. The Mellin transform $\fint_0^\infty  \xi^{2\nu+1}(\xi^2 + a)^{-1} d\xi$ can be defined as the analytic continuation of $\int_0^\infty  \xi^{2\nu+1}(\xi^2 + a)^{-1} d\xi$ from $\Re \nu <0$ or as a Hadamard finite part integral, see \sec{ss:fp}. Therefore, if we substitute $a$ by $\square_{h}=\square_{h(0)}$ we obtain formally
$$
\DN_{g}(\nu) = \cons (\square_{h(0)})^{\nu},
$$
in agreement  with the result proved by Enciso--González--Vergara in the static product case \cite{Enciso2017}. In the next sections we consider rigorously the general $x$-dependent case using a strategy inspired by the above formal computations.

\section{Hankel multipliers and paired Lagrangian distributions} \label{sec:hankel}

\subsection{Hankel transform on the half-line}\label{ss:hankel} Let $J_\nu(x)$ be the Bessel function of order $\Re\nu > -1$ (our focus is on the case $\Re \nu >0$, but the broader range  $\Re\nu > -1$  is sometimes helpful for analytic continuation arguments and formulas involving $J_{\nu-1}(x)$).  The function $\cJ_\nu(x):=x^\12  J_\nu(x)$ on $\rr_+$ satisfies the following bounds:
\beq\label{eq:bounds}
\bea
\module{\cJ_\nu(x)}&\leq C x^{\Re\nu+\12} \mbox{ for } 0\leq x \leq 1,\\
\module{\cJ_\nu(x)}&\leq C_\nu   \mbox{ for }  x\geq 1,
\eea
\eeq
where the first one is uniform in $\nu$ but not the second, see e.g.~\cite{Derezinski2017} for more precise asymptotics. Furthermore, $\overline{\cJ_\nu(x)}=\cJ_{\bar\nu}(x)$.

The  differential operator
$$
\p_x + (\nu-\textstyle\12) x^{-1}=  x^{-\nu+\12} \p_x x^{\nu-\12}   \in \Diff_\nu^{1}({\rr_+})
$$
is often called  {twisted derivative} of order $\nu-\12$. Its relevance comes from the fact  that
$$ 
\bea
N_\nu &=   - \p_x^2 + \big(\nu^2-{\textstyle\frac{1}{4}}\big){x^{-2}} 
=   (-\p_x + (\nu-\textstyle\12) x^{-1})(\p_x + (\nu-\textstyle\12) x^{-1}).  
\eea
$$
From the identities $\p_x( x^{\nu} J_\nu(x))=x^\nu J_{\nu-1}(x)$ and $\p_x( x^{-\nu} J_\nu(x))=-x^{-\nu} J_{\nu+1}(x)$ we get
\beq\label{ederiv}
\bea
&(\p_x + (\nu-\textstyle\12) x^{-1})\big( \cJ_\nu(x\xi)\big) =   \xi \,\cJ_{\nu-1}(x\xi), \\ 
&(-\p_x + (\nu-\textstyle\12) x^{-1})\big( \cJ_{\nu-1}(x\xi)\big) =   \xi \,\cJ_{\nu}(x\xi),
\eea
\eeq
for all $\xi>0$, 
hence $ N_{\nu,x} \big( \cJ_{\nu}(x\xi)\big)  = \xi^2 \cJ_{\nu}(x\xi) $.

For $\Re \nu > -1$, the Hankel and inverse Hankel transforms are given by
$$ 
\big(\cH_\nu f \big) (\xi) = \int_{0}^\infty \cJ_\nu(x\xi)f(\xi)d\xi, \quad \big(\cH_\nu^{-1} g \big) (x)=     \int_{0}^\infty \cJ_\nu(\xi x)g(x)dx,
$$
Then $\cH_\nu$, defined initially on $C_{\rm c}^\infty(\rr_+)$, extends to a bounded operator on  $L^2(\rr_+)$ satisfying $\cH_\nu^{-1}=\cH_\nu=\cH_{\bar\nu}^*$, see  \cite[Prop.~4.5]{Derezinski2017}. 
 Using \eqref{ederiv} one gets 
\beq\label{eq:diag}
\bea
\p_x + (\nu-\textstyle\12)x^{-1}   &= \cH_{\nu-1}^{-1} \circ \xi \circ \cH_\nu,\\
-\p_x + (\nu-\textstyle\12)x^{-1}    &= \cH_{\nu}^{-1} \circ \xi \circ \cH_{\nu-1},
\eea
\eeq
and $N_\nu= \cH_\nu^{-1} \circ \xi^2 \circ \cH_\nu$. 

\subsection{Regularized integrals of symbols}\label{ss:fp} We will come across expressions that are not necessarily integrable near zero or infinity, but which have a regularized integral in the following sense.

\begin{definition} \label{def:sy} Let $f\in L^1_{\rm loc}(\rr_+,\mathcal{F})$ be a locally Bochner integrable function on $\rr_+$ with values in a Fréchet space $\mathcal{F}$. Its \emph{Hadamard finite part integral}, or \emph{regularized integral} 
\beq\label{deffp}
\fint_{0}^\infty f(\xi) d\xi \defeq \fp_{R\to+\infty,\epsilon\to 0^+} \int_{\epsilon}^R f(\xi) d\xi  ,
\eeq
is defined as the constant term in the log-polyhomogeneous expansion of  $R\mapsto  \int_1^R f(\xi)d\xi$ for large $R$ (denoted by $\fint_{1}^\infty f(\xi) d\xi $), plus the   constant term in the log-polyhomogeneous expansion of  $\epsilon\mapsto  \int_0^\epsilon f(\xi)d\xi$ for small $\epsilon>0$ (denoted by $\fint_{0}^1 f(\xi) d\xi $), provided that the two log-polyhomogeneous expansions exist.
\end{definition} 



Of course, if $f\in L^1(\rr_+,\mathcal{F})$ then the regularized integral  \eqref{deffp} coincides with the Bochner integral $\int_{0}^\infty f(\xi) d\xi$. The basic examples are the regularized integrals 
\beq\label{eq:regi}
\fint_1^\infty \xi^\beta d\xi=-\frac{1}{1+\beta}, \quad \fint_0^1 \xi^\beta d\xi=\frac{1}{1+\beta}, \quad \fint_0^\infty \xi^\beta d\xi=0.
\eeq
for $\beta\neq -1$; for $\beta=-1$ all three regularized integrals are $0$. More generally if $f\in S^m_{\rm ph}(\rr_+)+ L^1(\rr_+)$ for some $m\in \rr$ then $R\mapsto  \int_1^R f(\xi)d\xi$ has a log-polyhomogeneous expansion and hence a well-defined regularized integral as follows from the arguments in e.g.~\cite[\S3]{paycha} (where the strictly analogous case of $\rr^d$ instead of $\rr_+$ is considered).  An analogous statement holds true for  functions that are poly-homogeneous near $0$ modulo $L^1(\rr_+)$. 

In the sequel we will need variants of the following computation:

\begin{example}\label{ex:had} For  $\module{\arg a}<\pi$ and $\nu\in \cc$, $\xi^{2\nu+1}(\xi^2+a)^{\frac{m}{2}}\in  S^{2\nu+m+1}_\ph(\rr_+)$, and if $1+\nu\notin -\nn_0$ and $-1-\nu-\frac{m}{2}\notin-\nn_0$   then
$$
\fint_0^\infty \xi^{2\nu+1}(\xi^2+a)^{\frac{m}{2}} d\xi =\12 a^{1+\nu+{\frac{m}{2}}}\, \frac{\Gamma(1+\nu) \Gamma(-1-\nu-\frac{m}{2})}{\Gamma(-\frac{m}{2})}.
$$
In particular for $m=-2$ and $\nu\notin\zz$,
$$
\fint_0^\infty \xi^{2\nu+1}(\xi^2+a)^{-1} d\xi = \12 a^\nu\, \Gamma(1+\nu) \Gamma(-\nu).
$$
This can be shown for instance by expressing the anti-derivative of $\xi^{2\nu+1}(\xi^2+a)^{-1} $ in terms of the Gauss ${}_{2} F_1$ hypergeometric function and then using well-known   asymptotics of the ${}_{2} F_1$ function at infinity. 
\end{example}

\subsection{Symbol spaces} Motivated by the previous chapter we introduce classes of Hankel multipliers with values in operators or Schwartz kernel of operators (more precisely, paired Lagrangian distributions). We start by introducing the relevant symbol spaces.

We first recall the symbol space used to define paired Lagrangian distributions.
For $q,d,k\in\nn$, $p,l\in\rr$, the space of product-type symbols ${S}^{p,l}(\rr^{q}; \rr^d,\rr^k)$ consists by definition of all smooth functions $c$ on $\rr^{q}\times \rr^d\times \rr^k$  for all $K\subset \rr^{q}$ compact and  all multi-indices $\alpha,\beta,\gamma$,
 \beq
 \big|\p^\alpha_\eta\p^{\beta}_\sigma \p_z^\gamma c(z;\eta,\sigma)\big|\leq C_{\alpha \beta\gamma} \bra \eta\ket^{p-\module{\alpha}}\bra \sigma\ket^{l-\module{\beta}}.
 \eeq
 The best constants $C_{\alpha \beta\gamma}$ define a family of seminorms which allows  to equip the space ${S}^{p,l}(\rr^{q}; \rr^d,\rr^k)$ with a Fréchet space topology. 

Next, we  consider symbols with values in ${S}^{p,l}(\rr^{q}; \rr^d,\rr^k)$ which have the following very particular behavior.

\begin{definition}\label{def:Sev} For $m\in \rr$, we define $S^{m}S^{p,l}( \rr^q;\rr_+,\rr^d,\rr^k)$  to be the space of all smooth functions 
$b$ on $\ei\times \rr^{q}\times \rr_+\times \rr^d\times \rr^k$  for all $K\subset \ei\times \rr^{q}$ compact, all $i,j\in\nn_0$ and all multi-indices $\alpha,\beta,\gamma$,
 \beq\label{def:bxz}
 \big|\p_\xi^i\p^\alpha_\eta\p^{\beta}_\sigma  \p_z^\gamma b(z;\xi,\eta,\sigma)\big|\leq C_{i\alpha \beta\gamma} \big(\bra \eta\ket^{p}\bra \sigma\ket^{l} + \xi^2\big)^{\frac{m-i}{2}} \bra \eta\ket^{-\module{\alpha}}\bra \sigma\ket^{-\module{\beta}}.
 \eeq
 The best constants $C_{i\alpha \beta\gamma}$ give a family of seminorms which define a Fréchet space topology.
\end{definition}

We  can also generalize the definition to $m,p,l\in \cc$ by replacing $m,p,l$ on the r.h.s.~of \eqref{def:bxz} by  $\Re m$, $\Re p$, $\Re l$. Note that for fixed $\xi$, each $b\in S^{m}S^{p,l}( \rr^q;\rr_+,\rr^d,\rr^k)$ is a product type symbol in  ${S}^{\frac{mp}{2},\frac{ml}{2}}(\rr^{q}; \rr^d,\rr^k)$.

We will often abbreviate $S^{m}S^{p,l}( \rr^q;\rr_+,\rr^d,\rr^k)$  by $S^{m}S^{p,l}$. We denote by $$S^{m}_\ph S^{p,l}( \rr^q;\rr_+,\rr^d,\rr^k)$$  or $S^{m}_\ph S^{p,l}$ the subclass of symbols in $S^{m}S^{p,l}$ that have a poly-homogeneous expansion in $\xi$.


\begin{lemma}\label{lem:bsm} Let $b\in S^{m} S^{p,l}( \rr^q;\rr_+,\rr^d,\rr^k)$, $\varphi\in L^\infty(\rr_+)$ and $\nu\in \cc$ such that  $1+\Re\nu\notin -\nn_0$ and $-1-\Re\nu-\frac{m}{2}\notin-\nn_0$. Let either $I=\clopen{1,\infty}$ and $\Re 2\nu+ 1 + m <  -1$, or $I=\clopen{0,1}$ and $\Re 2\nu+ 1  >  -1$. Then
\beq\label{bsm}
\int_I \xi^{2\nu + 1} b(z;\xi,\eta,\sigma)\varphi(\xi)d \xi  \in{S}^{p(1+\nu+\frac{m}{2}),l(1+\nu+\frac{m}{2})}(\rr^{q}; \rr^d,\rr^k).
\eeq
Moreover, if $\varphi\in L^2(\rr_+)$ then each semi-norm in $S^{p(\frac{3}{4}+\nu+\frac{m}{2}),l(\frac{3}{4}+\nu+\frac{m}{2})}(\rr^{q}; \rr^d,\rr^k)$ is bounded by a constant times $\norm{\varphi}_{L^2(\rr_+)}$. 
\end{lemma}
\begin{proof} We focus on the first case $I=\clopen{1,\infty}$, the other one being analogous.  Using the symbol estimate
$$
 \big|\p^\alpha_\eta\p^{\beta}_\sigma\p_z^\gamma b(z;\xi,\eta,\sigma)\big|\leq C_{ij\alpha \beta\gamma} \big(\bra \eta\ket^{p}\bra \sigma\ket^{l} + \xi^2\big)^{\frac{m}{2}} \bra \eta\ket^{-\module{\alpha}}\bra \sigma\ket^{-\module{\beta}}
 $$
and denoting the integral \eqref{bsm} by $c(z;\eta,\sigma)$, we get for $1+\nu\notin -\nn_0$ and $-1-\nu-\frac{m}{2}\notin-\nn_0$
 \beq\label{derp}
 \bea 
  \big|\p^\alpha_\eta\p^{\beta}_\sigma\p_z^\gamma c(z;\eta,\sigma)  | &\leq C_{\alpha\beta\gamma}  \bra \eta\ket^{-\module{\alpha}}\bra \sigma\ket^{-\module{\beta}}  \int_1^\infty  \xi^{2\Re\nu+1} \big(\bra \eta\ket^{p}\bra \sigma\ket^{l} + \xi^2\big)^{\frac{m}{2}}  \module{\varphi(\xi)}d\xi \\
  &= C_{\alpha\beta\gamma}  \bra \eta\ket^{-\module{\alpha}}\bra \sigma\ket^{-\module{\beta}}  (\bra \eta\ket^{p}\bra \sigma \ket^l )^{1+\Re \nu+\frac{m}{2}} \fantom \times \int_{\bra \eta\ket^{-p/2}\bra \sigma\ket^{-l/2}}^\infty  \xi^{2\Re\nu+1} \big( 1 + \xi^2\big)^{\frac{m}{2}}  \module{\varphi(\bra \eta\ket^{p/2}\bra \sigma\ket^{l/2} \xi)}d\xi \\
    &\leq C_{\alpha\beta\gamma}  \bra \eta\ket^{l(1+\Re\nu+\frac{m}{2})-\module{\alpha}}\bra \sigma\ket^{p(1+\Re\nu+\frac{m}{2})-\module{\beta}}  \fantom \times \| \varphi\|_\infty \int_{\bra \eta\ket^{-p/2}\bra \sigma\ket^{-l/2}}^\infty  \xi^{2\Re\nu+1} \big( 1 + \xi^2\big)^{\frac{m}{2}}  d\xi \\
        &\leq C_{\alpha\beta\gamma}  \bra \eta\ket^{l(1+\Re\nu+\frac{m}{2})-\module{\alpha}}\bra \sigma\ket^{p(1+\Re\nu+\frac{m}{2})-\module{\beta}}  \fantom \times \| \varphi\|_\infty \fp_{\epsilon\to 0} \int_\epsilon^\infty \xi^{2\Re\nu+1} \big( 1 + \xi^2\big)^{\frac{m}{2}}  d\xi \\
  &\leq  C_{\alpha\beta\gamma}'  \|\varphi\|_\infty  \bra \eta\ket^{l(1+\Re\nu+\frac{m}{2})-\module{\alpha}}\bra \sigma\ket^{p(1+\Re\nu+\frac{m}{2})-\module{\beta}},
  \eea 
\eeq
where the Hadamard finite part integral was computed in Example  \ref{ex:had}. This proves \eqref{bsm}. If  $\varphi\in L^2(\rr_+)$ then we can use the Cauchy--Schwarz inequality in place of the second inequality to get a bound in terms of
  $$
  \norm{ \varphi(\bra \eta \ket^{p/2} \bra \sigma \ket^{l/2}   \cdot )}_{L^2(\rr_+)}=\bra \eta \ket^{-p/4} \bra \sigma \ket^{-l/4} \norm{\varphi}_{L^2(\rr_+)}.
  $$ 
This produces   an additional $\bra \eta \ket^{-p/4} \bra \sigma \ket^{-l/4}$ factor which contributes to shifting  the orders correspondingly.
\end{proof}

\begin{remark}\label{rem:re} If  $b\in S^{m}_{\rm ph} S^{p,l}( \rr^q;\rr_+,\rr^d,\rr^k)$ then in Lemma \ref{lem:bsm}  we can replace the conditions $1+\Re\nu\notin -\nn_0$ and $-1-\Re\nu-\frac{m}{2}\notin-\nn_0$ by $1+\nu\notin -\nn_0$ and $-1-\nu-\frac{m}{2}\notin-\nn_0$. In fact the principal contribution can be computed  in a similar vein, with  an identity instead of an estimate for the absolute value (so $\nu$ appears instead of $\Re \nu$), and all other contributions are dealt with by using Lemma \ref{lem:bsm} in its original form. 
\end{remark}

\subsection{Paired Lagrangian distributions}\label{ss:pld} We recall standard material on paired Lagrangian distributions following mostly  \cite{Melrose1979,Greenleaf1990}, and we generalize some of it to Bessel symbols with values in paired Lagrangian distributions. Paired  Lagrangian distributions were  introduced first by Melrose--Uhlmann \cite{Melrose1979} in the case  intersection of codimension $k=1$, and then in full generality by Guillemin--Uhlmann \cite{Guillemin1981}.

If $Y$ is a (boundaryless) manifold of dimension $d$, we denote by $N^*\diag$ the conormal bundle of the diagonal in $(T^*Y\setminus \zero)\times (T^*Y\setminus \zero)$, i.e.
$$
N^*\diag  = \big\{ \big((y; \eta), (y; -\eta)\big) \big\}.
$$

We are particularly interested in intersecting pairs of Lagrangians $(\Lambda_0,\Lambda_1)$ such that $\Lambda_0=N^*\diag$ and $\Lambda_1$ is a flow-out Lagrangian.  

A useful model case on $Y=\rr^{d}$ is provided by the pair $(\widetilde \Lambda_0, \widetilde\Lambda_1)$ given by $\widetilde \Lambda_0=N^*\diag$ and
$$
\widetilde\Lambda_1 = \{ ((y; \eta), (y'; \eta')) \st y_2=y_2', \ \eta_1=\eta'_1=0, \ \eta_2=-\eta_2' \},
$$
 where the notation $y=(y_1,y_2) \in \rr\times \rr^{d}$ is used for points in $\rr^{d}$.  The Lagrangian $\widetilde\Lambda_1$ is the flow-out corresponding to the codimension $1$ involutive manifold $\{\eta_1=0\}$, i.e.~$\widetilde\Lambda_1$ consists of pairs of points $(y; \eta)$ and $(y'; \eta')$ which are connected by the Hamilton flow of $\eta_1$.  
  Let us recall that for $q\in\nn$, $p,l\in\cc$, the space of product-type symbols ${S}^{p,l}(\rr^{q}; \rr^d,\rr)$ consists by definition of all smooth functions $c$ on $\rr^{q}\times \rr^d\times \rr$  for all $K\subset \rr^{2d+1}$ compact,
 \beq
 \big|\p^\alpha_\eta\p^{\beta}_\sigma \p_z^\gamma c(z;\eta,\sigma)\big|\leq C_{\alpha \beta\gamma} \bra \eta\ket^{\Re p-\module{\alpha}}\bra \sigma\ket^{\Re l-\module{\beta}}.
 \eeq
  Then,  $I^{p,l}(\rr^d\times \rr^d; \widetilde{\Lambda}_0,\widetilde{\Lambda}_1)$ is the space of distributions $u$ on $\rr^d\times \rr^d$ of the form 
  \beq\label{eq:kyy}
  u(y,y')=\int e^{i((y_1-y_1'-s)\eta_1 + (y_2-y_2')\cdot \eta_2 + s \cdot \sigma)} c(y,y',s;\eta,\sigma) \,d\sigma \, ds \, d\eta 
  \eeq
  with $c\in{S}^{p+\12,l-\12}(\rr^{2d+1}; \rr^d,\rr)$, and the latter space induces a Fréchet topology on  $I^{p,l}(\rr^d\times \rr^d; \widetilde{\Lambda}_0,\widetilde{\Lambda}_1)$.  More generally, for $m\in\cc$ we can  consider  the space 
  \beq\label{eq:Ssev}
   S^{m} I^{p,l} (\rr^d\times \rr^d;\rr_+,\widetilde{\Lambda}_0,\widetilde{\Lambda}_1)
  \eeq
    of  Bessel symbols with values in paired Lagrangian distributions which are obtained by the integral formula \eqref{eq:kyy}  applied to $c\in  S^{m} {S}^{p+\12,l-\12} (  \rr^{2d+1}; \rr_+,\rr^d,\rr)$.
  
  If  $(\Lambda_0,\Lambda_1)$ is a cleanly intersecting pair of Lagrangian and the intersection is of codimension $1$, there exists a canonical transformation $\chi:(T^*Y\setminus \zero)\times(T^*Y\setminus \zero)$ which maps $\chi(\Lambda_0)\subset \widetilde\Lambda_0$ and $\chi(\Lambda_1)\subset \widetilde\Lambda_1$. This allows one to define $I^{p,l}(Y\times Y; {\Lambda}_0,{\Lambda}_1)$ in greater generality as follows. 
  
  \begin{definition} The space $I^{p,l}(Y\times Y; {\Lambda}_0,{\Lambda}_1)$ consists of distributions which are locally finite sums of terms of the form $F_i u_i(\cdot,\cdot)$, where each $F_i$ is a Fourier integral operator associated with a canonical transformation $\chi_i$ as above and $u_i\in I^{p,l}(\rr^d\times \rr^d; \widetilde{\Lambda}_0,\widetilde{\Lambda}_1)$.  The space  $S^{m} I^{p,l}(Y\times Y;\rr_+, {\Lambda}_0,{\Lambda}_1)$  is defined analogously with $u_i\in    S^{m} I^{p,l} (\rr^d\times \rr^d;\rr_+,\widetilde{\Lambda}_0,\widetilde{\Lambda}_1)$ instead (and $F_i$ independent on $\xi$).
  \end{definition}     

In the $\xi$-independent version, the fundamental fact  is that if $u\in I^{p,l}(Y\times Y; {\Lambda}_0,{\Lambda}_1)$ then
$$
u \in I^{p+l}(Y\times Y; {\Lambda}_0\setminus (\Lambda_0\cap{\Lambda}_1)) \mbox{ and } u\in I^{p}(Y\times Y; {\Lambda}_1\setminus (\Lambda_0\cap{\Lambda}_1)).
$$
Therefore $u$ has a symbol $\sigma_0(u)$ as a Fourier integral operator in the former space, and another one $\sigma_1(u)$   in the sense of the latter space. 

   We will often occasionally $S^{m} I^{p,l}( Y\times Y;\rr_+, {\Lambda}_0,{\Lambda}_1)$ by $S^m I^{p,l}$ and write $S^{m}_{\rm ph} I^{p,l}( Y\times Y;\rr_+, {\Lambda}_0,{\Lambda}_1)$ or $S^m_{\rm ph} I^{p,l}$ for classes with $\xi$-poly-homogeneous symbols.

   \begin{remark}\label{rem:ano} For fixed  $\xi\in\rr_+$, each $u\in S^{m} I^{p,l}(Y\times Y;\rr_+, {\Lambda}_0,{\Lambda}_1)$ is a paired Lagrangian distribution in $I^{\frac{mp}{2}+\frac{m}{4}-\frac{1}{2},\frac{ml}{2}-\frac{m}{4}+\frac{1}{2}}(Y\times Y; {\Lambda}_0,{\Lambda}_1)$.
Furthermore if $m\leq 0$, using  the  estimate  $\big(\bra \eta\ket^{p}\bra \sigma\ket^{l} + \xi^2\big)^{\frac{m-i}{2}}\leq C \big(\bra \eta\ket^{p}\bra \sigma\ket^{l} \big)^{\frac{m}{4}-\frac{i}{2}} \bra \xi \ket^{\frac{m}{2}}$ on the level of symbols, we can continuously embed $S^{m} I^{p,l}$ in a space of functions  with values in $I^{\frac{mp}{4}+\frac{m}{8}-\frac{1}{2},\frac{ml}{4}-\frac{m}{8}+\frac{1}{2}}$, with  $O(\bra \xi \ket^{\frac{m}{2}})$ dependence on $\xi\in\rr_+$.
   \end{remark}

   \begin{lemma}\label{lem:bsm2} Let $u\in S^{m}_{\rm ph} I^{p-\12,l+\12}( \rr^q;\rr_+,\rr^d,\rr^k)$, $\varphi\in L^\infty(\rr_+)$ and $\nu\in \cc$ such that  $1+\nu\notin -\nn_0$ and $-1-\nu-\frac{m}{2}\notin-\nn_0$. Let either $I=\clopen{1,\infty}$ and $\Re 2\nu+ 1 + m <  -1$, or $I=\clopen{0,1}$ and $\Re 2\nu+ 1  > -1$. Then
   \beq\label{eq:key1}
  v:= \int_I \xi^{2\nu + 1} u(\xi)\varphi(\xi) d \xi   \in {I}^{(p+\12)(1+\nu+\frac{m}{2})-\12,(l-\12)(1+\nu+\frac{m}{2})+\12}(Y\times Y; {\Lambda}_0,{\Lambda}_1).    \eeq
  Moreover, if $\varphi\in L^2(\rr_+)$ then each semi-norm of $v$ in ${I}^{(p+\12)(\frac{3}{4}+\nu+\frac{m}{2})-\12,(l-\12)(\frac{3}{4}+\nu+\frac{m}{2})+\12}$ is bounded by a constant times $\norm{\varphi}_{L^2(\rr_+)}$.  
      Furthermore, if $\Re \nu +1 + \frac{m}{2}\geq 0$ then for $i=0,1$,
$
      \sigma_i (v)=  \int_I  \xi^{2\nu+1} \big(\sigma_i(b)\big)(\xi) d\xi.  
   $
   \end{lemma}
\begin{proof} The first claim follows directly from Lemma \ref{lem:bsm} and the paired Lagrangian classes definitions (using also Remark \ref{rem:re} to weaken the assumption on $\nu$).  If in addition  $\Re \nu+ 1 + \frac{m}{2}\geq 0$ then decreasing $p+l$ decreases the added orders of \eqref{eq:key1}. Similarly, decreasing $p$ decreases the  first order of  \eqref{eq:key1}. Therefore, only the principal symbols $\sigma_i(u)$ of $u$ contribute to the principal symbols $\sigma_i(v)$. 
\end{proof}



Our main case of interest is $m=-2$,  $p=\12$  and $l=\frac{3}{2}$,   which gives $v\in I^{\nu-\12,\nu+\12}(Y\times Y; {\Lambda}_0,{\Lambda}_1)$, consistently with complex powers constructions of \cite{Antoniano1985,joshi}. 

\subsection{Parameter-dependent Klein--Gordon operators}

Let $h$ now be a  Lorentzian metric on a $d=n-1$-dimensional manifold  $Y$.

Let $\Sigma=\{ (y;\eta) \in T^*Y \st  \module{\eta}_{h}^2=0 \}$ be the characteristic set of $\square_{h}$. Then $\Sigma$ has four orientations corresponding to four choices of types of parametrices corresponding to respectively forward, backward, Feynman and anti-Feynman parametrices. 

We use the notation from \sec{ss:pld} with $d=n-1$,  $\Lambda_0=N^*\diag$, and $\Lambda_{1}$ being the flowout of $\Sigma$ in a chosen orientation. For the sake of definiteness we take the orientation corresponding to the forward (also called retarded) parametrix.        

In  \cite[\S6]{Melrose1979}, Melrose--Uhlmann construct in particular  a parametrix for inverses of the Klein--Gordon operator $\square_{h}+\xi^2$ for fixed $\xi$ using the calculus of paired Lagrangian distributions. Here we need a similar result, but with uniform control in $\xi>0$.

 
 \begin{proposition}\label{prop:parKG}  
Let $(\square_{h}+\xi^2)^{-1}$ be the  retarded propagator of $(\square_{h}+\xi^2)$.    Then
 \beq
 (\square_{h}+\xi^2)^{-1}\in     S^{-2}_\ph I^{\12,\frac{3}{2}}(Y\times Y;\rr_+, {\Lambda}_0,{\Lambda}_1),
 \eeq
 and its $\xi$-valued principal symbol at $\Lambda_0\setminus\Lambda_1$ is
$$
\sigma_0\big((\square_{h}+ \xi^2)^{-1} \big)(y;\eta)=(\module{\eta}_{h}^2 +\xi^2)^{-1}.
$$
 \end{proposition} 
 \begin{proof} The proof is divided into two steps.
 
\step{1} In the first step, for $\xi>0$ we consider the special case of the constant coefficients Klein--Gordon operator $$\square_{h_0}+\xi^2=-D_t^2+D_{\bar y}^2+\xi^2$$ on $\rr^d$, where $h_0$ is the flat metric and we write $y=(t,\bar y)$ for the variables  on $\rr^d$. We denote by $(\square_{h_0}+\xi)^{-1}$ the retarded inverse of $\square_{h_0}+\xi^2$. Its Schwartz kernel is a multiple of the translation-invariant distribution $u(t-t',\bar{y}-\bar{y}')$, where 
 \beq\label{eq:uty}
 u(t,\bar{y})=\int_{\rr^d} e^{-i (t \tau+\bar{y} \bar{\eta})}\big(-(\tau-i0)^{-1}+\bar{\eta}^2 +\xi^2  \big)^{-1} d\tau\,d\bar\eta,
 \eeq
 $\eta=(\tau,\bar{\eta})$ being the dual variables. In a conic neighborhood of the component $\{ \tau = \pm \module{\bar\eta}\}$ of the characteristic set of $\square_{h_0}$, the canonical transformation
\beq\label{eq:yyee}
y_1=t, \quad 
y_2= \bar{y}\pm\frac{\bar\eta}{\module{\bar\eta}} t, \quad
\eta_1=\tau \mp \module{\bar\eta}, \quad
\eta_2= \bar\eta
\eeq
maps the Lagrangian pair $(\Lambda_0,\Lambda_1)$ (strictly speaking, its image under  $(y,y';\eta,\eta')\mapsto (y-y';\eta+\eta')$) to the Lagrangian pair $(\widetilde{\Lambda}_0,\widetilde{\Lambda}_1)$ given by
$$
\bea
\widetilde{\Lambda}_0= \{(y_1,y_2;\eta_1,\eta_2)\st y_1=0, \ y_2=0 \}, \quad
\widetilde{\Lambda}_1= \{(y_1,y_2;\eta_1,\eta_2)\st y_2=0, \ \eta_1=0 \}.
\eea
$$
By slight abuse of notation we also write $y=(y_1,y_2)$ and $\eta=(\eta_1,\eta_2)$. Let $\chi_\pm\in \cf(\rr^{d},[0,1])$ be equal $1$ on $\{ \tau=\pm\module{\bar \eta}, \ \tau^2+\bar \eta^2\geq 1\}$ and supported in a small neighborhood of that set. We consider the microlocalized expression
\beq\label{eq:pmi}
\sum_{\pm}   \int_{\rr^d} e^{-i (t \tau+\bar{y} \bar{\eta})}\chi_\pm(\tau,\bar\eta)\big(-(\tau-i0)^{2}+\bar{\eta}^2 +\xi^2  \big)^{-1} d\tau\,d\bar\eta
\eeq
and define  $(\square_{h_0}+\xi^2)^{\inv}$ through corresponding  translation invariant Schwartz kernel. Then, $(\square_{h_0}+\xi^2)^{\inv}$ is a parametrix of $\square_{h_0}+\xi^2$ (with $\xi$-independent error), and we want to prove that it is a $\xi$-dependent paired Lagrangian distribution. We apply the transformation \eqref{eq:yyee} to   \eqref{eq:pmi}, and find that  the two respective summands  become
$$
\bea
& \int_{\rr^d} e^{-i (y_1 \eta_1+y_2 \eta_2)}\chi_\pm(\tau(\eta),\bar\eta(\eta)) \big( \eta_1(\eta_1+2\module{\eta_2}) +\xi^2  \big)^{-1} d\eta \\
& = \int_{\rr^{d+1}} e^{-i (y_1 \eta_1+y_2 \eta_2+s(\eta_1-\sigma))}\psi_\pm(\eta,\sigma) \big( \sigma (\eta_1+2\module{\eta_2}) +\xi^2  \big)^{-1} d\eta \, d\sigma
\eea
$$
for some smooth $\psi$ which equals $1$ near $\{\eta_1=0, \ \sigma^2+\eta^2\geq 1\}$ and is supported in a small neighborhood of that set. Taking into account the support properties of $\psi$, we can see that
$$
c_\pm(y,s;\xi,\eta,\sigma)\defeq \psi_\pm(\sigma,\eta)\big( \sigma (\eta_1+2\module{\eta_2}) +\xi^2  \big)^{-1}
$$
 is a symbol   in $S^{-2}_{\rm ph}S^{1,1}(  \rr^{d+1}; \rr_+,\rr^d,\rr)$. In consequence, 
 $$
 (\square_{h_0}+\xi^2)^{\inv}\in S^{-2}_\ph I^{\12,\frac{3}{2}} ( \rr^d\times \rr^d;\rr_+,{\Lambda}_0,{\Lambda}_1).
 $$
 
 \step{2} For each $q\in T^*Y\setminus \zero$ we can find FIOs $F,F'$ associated with  canonical transformations mapping the respective Lagrangian pairs $\Lambda_0,\Lambda_1$ for $h$ and $h_0$ one to another, such that $F'$ is microlocally the inverse of $F$ at $q$ and $R:= F \square_{h_0}- \square_{h} F$ satisfies $q\notin \wf'(R)$, and $q'\notin \wf'(R)$ for any $q'$ in the flowout of $q$ if $q$ is characteristic. 
 
 This follows from \cite[Prop.~5.1]{joshi} in the case when $h$ is  ultra-static, but the proof applies to the general case with only minor modifications that we briefly explain. The first step is \cite[Lem.~5.2]{joshi} which constructs a symplectomorphism in a conic neighborhood of $q$ (containing the whole bicharacteristic through $q$), such that the pullback of  $\sigma_0(\square_{h})$ is  $\sigma_0(\square_{h_0})$. The key step is to factorize $\sigma_0(\square_{h})=\eta_1 \eta_2$ where $\eta_1,\eta_2$ are homogeneous of degree one, Poisson commute and have non-vanishing differentials near the characteristic variety. This is straightforward in the static case; more generally we can use for instance the coordinates and the factorization in \cite[(6.10)]{GOW} locally near $q$, translated to the level of symbols (note that global aspects and lower order terms do not matter here), and then extend via the bicharacteristic flow. The remaining steps of the proof of \cite[Prop.~5.1]{joshi} can be  repeated verbatim.
 
   Next, we can write
 \beq\label{FFp}
F' (\square_{h}+\xi^2)^{-1} F =  F' F (\square_{h_0}+\xi^2)^{-1} + F'(\square_{h} + \xi^2)^{-1} R  (\square_{h_0} +\xi^2 )^{-1}.
\eeq
By Step 1 the first summand is in  $S^{-2}_\ph I^{\12,\frac{3}{2}}$. Setting $A=\square_{h}-\one$, the second summand can be written as 
 \beq\label{FFp2}
 F'\bigg(-\sum_{j=1}^N (-1)^j\bra\xi\ket^{-2j}A^{{j-1}} +(-1)^N \bra \xi \ket^{-2N} A^N (A+\bra\xi\ket^2)^{-1} \bigg) R  (\square_{h_0} +\xi^2 )^{-1}
\eeq
for arbitrary $N\in\nn$.  The first $N$ terms are  in $S^{-2}_\ph I^{\12,\frac{3}{2}}$, and the last term is $O(\bra \xi \ket^{-2N})$ with values in  operators at $q$ by microlocal mapping properties of $R$ and $(A+\bra\xi\ket^2)^{-1}$. The principal symbol is deduced from \eqref{FFp}. 
 \end{proof}

 For ease of notation we will write
 \beq\label{eq:notI}
I^{\alpha}(Y;\Lambda_0,\Lambda_1):={I}^{\frac{\alpha}{2} - \12,\frac{\alpha}{2}+\12}(Y\times Y; {\Lambda}_0,{\Lambda}_1).
 \eeq
 Note in particular $\Diff^\alpha(Y)\subset I^{\alpha}(Y;\Lambda_0,\Lambda_1)$.

 \subsection{Regularized integrals and forward complex powers}\label{ss:reg}
  
 Given a globally hyperbolic spacetime $(Y,h)$ let us introduce the following parameter-dependent classes of paired Lagrangian distributions on $Y$ for $m\in \nn$ and $s\in \nn_0^m$:
 \beq\label{defIms}
\cI^{m,s}:= \left\{ (\wo+\xi^2)^{-1} L_1 \circ \cdots  \circ(\wo+\xi^2)^{-1} L_m   \st L_k\in \Diff^{s_k}(Y), \ k=1,\dots,m \right\},
 \eeq
and by convention $\cI^{0,0}$ consists of multiples of the identity. The compositions with forward inverses $(\wo+\xi^2)^{-1}$ are well-defined on spaces on future supported Sobolev spaces. Note that all the $L_k$ are $\xi$-independent.  From the identity
$$
(\wo+\xi^2)^{-1}=\xi^{-2} - \xi^{-2} \wo(\wo+\xi^2)^{-1}
$$  
we deduce for each $\alpha\in\cc$ the inclusion
\beq\label{eq:rec1}
\bea
\xi^\alpha\cI^{m,s}
&\subset \xi^{\alpha-2} \Diff^{s_m}(Y)\circ  \cI^{m-1,s|_{m-1}} + \xi^{\alpha-2}\Diff^{2}(Y) \circ \cI^{m,s}\\
&\subset \xi^{\alpha-2} I^{s_m}(Y;\Lambda_0,\Lambda_1)\circ  \cI^{m-1,s|_{m-1}}+ \xi^{\alpha-2}  I^{2}(Y;\Lambda_0,\Lambda_1) \circ \cI^{m,s},
\eea
\eeq
where $s|_{m-1}\in \nn_0^{m-1}$ is the multi-index $s$ with the $m$-th component removed. Similarly, using the identity 
$$
(\wo+\xi^2)^{-1} = \wo^{-1} - \wo^{-1}\xi^2 (\wo+\xi^2)^{-1}
$$
where $\wo^{-1}\in I^{-2}(Y;\Lambda_0,\Lambda_1)$, we can deduce 
\beq\label{eq:rec2}
\bea
\xi^\alpha\cI^{m,s}\subset \xi^\alpha I^{s_m-2}(Y;\Lambda_0,\Lambda_1)\circ  \cI^{m-1,s|_{m-1}}+ \xi^{\alpha+2}  I^{-2}(Y;\Lambda_0,\Lambda_1) \circ \cI^{m,s},
\eea
\eeq

Notice that if  $b\in \xi^\alpha\cI^{m,s}$  for  $\Re \alpha-2m<-1$, resp.~$\Re \alpha>-1$, then it is integrable in $\xi$ on the interval $\clopen{1,+\infty}$, resp.~$\open{0,1}$. More generally, we can  compute the regularized integrals
$$
 \fint_1^\infty \xi^\alpha b(\xi) d\xi, \quad \fint_0^1 \xi^\alpha b(\xi) d\xi, \quad \fint_0^\infty \xi^\alpha b(\xi) d\xi 
$$
in the following way.  For the first one, we use \eqref{eq:rec1} repeatedly to reduce the computation to terms that are either integrable in $\xi$ (and just integrate them over $\clopen{1,+\infty}$), or are in $\xi^{\beta}I^{r}(Y;\Lambda_0,\Lambda_1) \circ \cI^{0,0}=\xi^{\beta}I^{r}(Y;\Lambda_0,\Lambda_1) $ for  some $\beta,r$. Similarly for the second one, we use \eqref{eq:rec2}  to reduce to terms that are either integrable in $\xi$ (and integrate them over $\clopen{0,1}$), or are in $\xi^{\beta}I^{r}(Y;\Lambda_0,\Lambda_1)$. In this way we are reduced to computing regularized integrals of $\xi^\beta$ which are given by \eqref{eq:regi}. 

\begin{proposition}\label{prop:fint} Let $\varphi\in L^\infty(\rr_+)$, $\alpha \notin 2\zz+1$ and $b\in \xi^\alpha I^{r}(Y;\Lambda_0,\Lambda_1) \circ\cI^{m,s}$ for some $r\in \rr$, $m\in \nn_0$ and $s\in \nn_0^m$.
Then
\beq
\fint_0^\infty  b(\xi)\varphi(\xi) d\xi  \in I^{\alpha+r-2m+|s|+1}(Y;\Lambda_0,\Lambda_1).
\eeq
Moreover, if $\varphi\in L^2(\rr_+)$ then each semi-norm in $I^{\alpha+r-2m+|s|+\frac{3}{4}}(Y;\Lambda_0,\Lambda_1)$ is bounded by a constant times $\norm{\varphi}_{L^2(\rr_+)}$.  
 \end{proposition} 
 \begin{proof} Observe that iterating the inclusions \eqref{eq:rec1} or \eqref{eq:rec2} (applied initially to $b$) used to define the regularized integrals  yields always sums of terms that are each in $\xi^{\alpha'} I^{r'}(Y;\Lambda_0,\Lambda_1) \circ\cI^{m',s'}$ for some $\alpha', r',m',s'$ such that $\alpha'+r'-2m'+|s'|=\alpha+r-2m+|s|$ and $\Re \alpha'\notin 2\zz+1$. In consequence, we are reduced to proving
 \beq\label{eq:intij}
 \int_1^\infty  b(\xi) \varphi(\xi) d\xi  \in I^{\alpha+r-2m+|s|+1}(Y;\Lambda_0,\Lambda_1)
 \eeq
in the integrable case $\Re \alpha-2m<-1$ and the analogous statement for the integral on $\open{0,1}$ in the integrable range $\Re \alpha>-1$. We focus on \eqref{eq:intij}, the second case being analogous. 
 
 The $I^{r}(Y;\Lambda_0,\Lambda_1)$ part does not depend on $\xi$ and is handled easily, so we disregard it in the sequel.  We claim that without loss of generality we can assume that $s=(0,\dots,0,|s|)$, i.e., assume $b$ is of the form $b\in(\wo+\xi^2)^{-m}\Diff^{|s|}(Y)$. Indeed, if it is not the case then we can commute the differential operators $L_k$ to the right using the resolvent identity $[(\wo+\xi^2)^{-1},L_k]= (\wo+\xi^2)^{-1} [L_k,\wo] (\wo+\xi^2)^{-1}$ and reiterate as many times as wanted. Each iteration on an element of $\cI^{m',s'}$ produces a term with one of the differential operators shifted to the right,   plus a commutator term in $\cI^{m'+1,s''}$ with $|s''|=|s'|+1$. Thus by iterating as many times as needed, for any  $N\in \nn$ we get that $b\in (\wo+\xi^2)^{-m}\Diff^{|s|}(Y)$ modulo a term in $\cI^{m+N,S}$ where $|S|=|s|+N$. The latter term is bounded in $\xi$ with values in $I^{-2(m+N)+|s|+N}(Y;\Lambda_0,\Lambda_1)= I^{-2m+|s|-N}(Y;\Lambda_0,\Lambda_1)$, so the corresponding contribution to the integral  \eqref{eq:intij} is in    $I^{-2m+|s|-N}(Y;\Lambda_0,\Lambda_1)$ and is therefore inessential as $N$ can be taken as large as wanted. 
 
 Now if $b\in(\wo+\xi^2)^{-m}\Diff^{|s|}(Y)$, the  $\Diff^{|s|}(Y)$ part is easily taken into account, so we disregard it in the sequel and assume $b(\xi)=(\wo+\xi^2)^{-m}$. The proof of Proposition \ref{prop:parKG} generalizes directly and gives $b\in S^{-2m} I^{\12,\frac{3}{2}}(Y\times Y;\rr_+, {\Lambda}_0,{\Lambda}_1)$. Thus, by Lemma \ref{lem:bsm2} we obtain 
 \beq\label{tpxalp}
 \int_1^\infty  \xi^\alpha (\square_{h_0} +\xi^2 )^{-m} \varphi(\xi)  d\xi  \in  {I}^{\frac{\alpha}{2}-m + \12,\frac{\alpha}{2}-m+\frac{3}{2}}(Y\times Y; {\Lambda}_0,{\Lambda}_1)
\eeq
and the latter space is $I^{\alpha-2m+1}(Y;\Lambda_0,\Lambda_1)$ by definition. The  assertion concerning the case $\varphi\in L^2(\rr_+)$ also follows from  Lemma \ref{lem:bsm2}.
 \end{proof}
 
 \medskip
 
Complex powers of $\square_h$ arise as a particular case of this construction.

 \begin{definition}\label{def:cp}  For $\nu \notin \zz$, the (forward) \emph{complex powers} of   $\square_{h}$  are the paired Lagrangian distributions 
\beq
\bea
\big(\square_{h}\big)^{\nu} & \defeq  \frac{2}{\Gamma(1+\nu)\Gamma(-\nu)}   \fint_0^\infty  \xi^{2\nu+1} (\square_{h}+\xi^2)^{-1} d\xi \fantom \,\in  I^{2\nu}(Y\times Y; {\Lambda}_0,{\Lambda}_1)= {I}^{\nu - \12,\nu+\12}(Y\times Y; {\Lambda}_0,{\Lambda}_1).
\eea
\eeq
 \end{definition} 
 
In the case of ultra-static spacetimes a definition of forward complex powers of $\square_h$ was proposed by Joshi, who also proved that they are a family of paired Lagrangian distributions that enjoys good composition properties \cite{joshi}. As we will see, the definition in \cite{joshi} actually extends to general globally hyperbolic spacetimes and we show that it coincide with ours. 

The construction in \cite{joshi} first introduces an auxiliary space variable $r$ and considers the wave operator $\square_h+D_r^2$  on $Y\times \rr_r$. By translation invariance in $r$ the Schwartz kernel of the forward inverse is of the form $(\square_h+D_r^2)^{-1}(y,y',r,r')=K(y,y',r-r')$. Joshi defines  $\big(\square_{h}\big)^{\nu}$ as the family of operators whose Schwartz kernel is the analytic continuation from $\Re \nu\ll 0$ of 
\beq 
2 e^{i\pi (\nu+1)} \pi_* \big(\chi_+^{-2\nu-2}(r) K(y,y',r)\big),
\eeq
where $\pi_*$ is the push-forward of the projection $\pi:(y,y',r)\mapsto (y,y')$ and $z\mapsto \chi_+^z$ is the family of distributions defined by analytic continuation of $\chi_+^z(x)= \theta(x)x^z \Gamma(1+z)$ ($\theta(x)$ being the Heaviside step function). The push-forward is well-defined thanks to support properties of $K$, namely, for $(y,y',r)$ in a compact set, it is compactly supported in $r$.

 \begin{proposition} For $\Re \nu\ll 0$,
$
 \big(\square_{h}\big)^{\nu}(y,y')= 2 e^{i\pi (\nu+1)} \pi_* \big(\chi_+^{-2\nu-2}(r) K(y,y',r)\big).
 $
 \end{proposition}
 \begin{proof} Writing $K(y,y',r)=\frac{1}{2\pi}\int_{-\infty}^\infty e^{i r \xi} (\square_h+\xi^2)^{-1}(y,y')d\xi$, the r.h.s.~equals
 $$
 \bea {}
&\frac{ e^{i\pi (\nu+1)}}{\pi} \pi_*\left( \int_{-\infty}^\infty e^{ir \xi } \chi_+^{-2\nu-2}(r) (\square_h+\xi^2)^{-1}(y,y')d\xi\right)\\
&\frac{e^{i\pi (\nu+1)}}{\pi} \lim_{\varepsilon\to 0^+}\int_{-\infty}^\infty \int_{-\infty}^\infty e^{-\varepsilon|r|} e^{ir \xi } \chi_+^{-2\nu-2}(r) (\square_h+\xi^2)^{-1}(y,y')d\xi dr\\
&= \frac{i e^{i\pi (\nu+1)}}{\pi} \int_{-\infty}^\infty (-\xi-i 0)^{1+2\nu} (\square_h+\xi^2)^{-1}(y,y') d\xi,
\eea
$$
which using evenness of $(\square_h+\xi^2)^{-1}(y,y')$ in $\xi$ can be rewritten as
$$
\frac{i e^{i\pi (\nu+1)} (1-e^{2\pi i\nu})}{\pi} \lim_{\delta\to 0^+}\int_{0}^\infty (-\xi-i \delta)^{1+2\nu}(\square_h+\xi^2)^{-1}(y,y') d\xi. 
 $$
 This is easily seen to coincide with the regularized integral \eqref{def:cp} using the identity  $\Gamma(1+\nu)\Gamma(-\nu)=-\frac{\pi}{\sin \pi \nu}$ and the characterizations of the family $(\xi+ i 0 )^{1+2\nu}$.
 \end{proof}\smallskip

 \subsection{Hankel multipliers and boundary asymptotics} 
 
Let us recall that $\cI^{m,s}$ are the $\xi$-dependent classes of paired Lagrangian distributions on $Y$ defined in \eqref{defIms} by composing $m$ times $(\wo+\xi^2)^{-1}$ and differential operators on $Y$ of order $\leq |s|$ in total. 

For $l\in \cc$, denote 
$
S^l(\rr_+^\circ):= \{ \varphi \in S^l(\rr_+) \st \supp \varphi \cap \{0\} = \emptyset \}
$.

We now set for $\ell\in\rr$,
\beq\label{eq:defIl}
\cI^{\ell} := S^0(\rr_+^\circ){\rm-span}  \bigcup_{\Re \alpha-2m+|s|+1=\ell, \Re \alpha-2m <-1} \xi^{\alpha} \cI^{m,s},
\eeq
where $\alpha\in \cc$,  $s\in \nn_0^m$ and $m\in\nn_0$.  In this way,
  Proposition \ref{prop:fint} gives  that $\fint_0^\infty b(\xi) d\xi \in I^\ell(Y;\Lambda_0,\Lambda_1)$ for $b\in \cI^{\ell}$ if $\Re\ell\notin \zz$.

 Let $\Re \nu,\Re \mu>0$. For $b\in \xi^\alpha \cI^{m,s}$, the \emph{Hankel multiplier} $\cH_\nu^{-1} b(\xi) \cH_\mu$ is well-defined through a convergent integral for instance if $\Re\alpha<2m-1$ and $b$ is supported away from $\xi=0$ (or more generally, if $\Re\alpha>-2-\Re (\nu+\mu)$), and its (operator-valued) Schwartz kernel is 
 $$
 B(x,x')=\int_0^\infty \cJ_\nu(x\xi)b(\xi)\cJ_\mu(x'\xi)d\xi. 
 $$
 More generally, the definition extends to other values of $\alpha$ if $B(x,x')$ is interpreted as an operator-valued distribution.   
 
 For $\ell\in \rr$, let $M_{\nu,\nu+1}^{\ell,1}(\rr_+)$ be the linear span of Hankel multipliers supported away from $\xi=0$ defined with either $\cJ_{\nu}$, $\cJ_{\nu+1}$, or both, i.e.
 $$
 M_{\nu,\nu+1}^{\ell,1}(\rr_+) := {\rm span} \left\{ \cH_{\nu+i} b(\xi)\cH_{\nu+j} \st b \in   \cI^\ell, \ i,j\in\{0,1\} \right\}.
 $$
 Then, by setting
 $$
 M_{\nu,\nu+1}^{\ell}(\rr_+):= \{  B_1 \circ \dots \circ B_k \st B_i \in M_{\nu,\nu+1}^{\ell_i,1}(\rr_+), \ \ell_1+\cdots + \ell_k = \ell, \ k\in\nn   \}, 
 $$
 we get a filtered algebra in the sense that
 $$
 M_{\nu,\nu+1}^{\ell} (\rr_+) \circ M_{\nu,\nu+1}^{\ell'}(\rr_+) \subset M_{\nu,\nu+1}^{\ell+ \ell'}(\rr_+).
 $$
 The reason for considering composition of Hankel multipliers with different parameters $\nu$ and $\nu+1$ is the following lemma:

 \begin{lemma}\label{lem:NN} For  $\ell\in \rr$,
 \beq\label{eq:psic}
 M^\ell_{\nu,\nu+1}(\rr_+)\circ    x \subset M^{\ell-1}_{\nu,\nu+1}(\rr_+) + x  \circ M^\ell_{\nu,\nu+1}(\rr_+), 
 \eeq
 where $x$ is understood as a  multiplication operator.
 \end{lemma}
 \begin{proof}  We first observe that $\xi^{-1}:S^0(\rr_+^\circ)\xi^\alpha \cI^{m,s}\to S^0(\rr_+^\circ)\xi^{\alpha-1}\cI^{m,s}$ and $$\p_\xi:S^0(\rr_+^\circ)\xi^\alpha \cI^{m,s}\to S^0(\rr_+^\circ)\xi^{\alpha-1}\cI^{m,s}+S^0(\rr_+^\circ)\xi^{\alpha+1}\cI^{m+1,(0,s)},$$ hence
 \beq\label{xipmaps}
\xi^{-1},\p_\xi : \cI^\ell\to \cI^{\ell-1}
\eeq
 for all $\ell\in \rr$.
 
 Now for any $b\in \cI^{m_i}$, using the two identities \eqref{eq:diag} with $\nu+1$ instead of $\nu$ (and with the role of $x$ and $\xi$ reversed) we obtain 
 \beq\label{eq:hankelx}
 \bea
 &\cH_\nu^{-1} b(\xi)\cH_{\nu} \circ  x=  \cH_\nu^{-1} (-\p_\xi b(\xi))  \cH_{\nu+1} + x \circ \cH_{\nu+1}^{-1}  b(\xi)\cH_{\nu+1},  \\
 &\cH_\nu^{-1} b(\xi) \cH_{\nu+1} \circ  x=  \cH_\nu^{-1} \big( (\p_\xi + (2\nu+1) \xi^{-1}) b(\xi) \big) \cH_{\nu} + x \circ \cH_{\nu+1}^{-1} (-b(\xi)) \cH_{\nu} \\
 \eea
 \eeq
 and similar identities for $\cH_{\nu+1}^{-1} b(\xi)\cH_{\nu} \circ  x$ and $\cH_{\nu+1}^{-1} b(\xi) \cH_{\nu+1} \circ  x$. In view of \eqref{xipmaps} these four identities give
 \beq\label{eq:Mmmi}
 M^{\ell_i,1}_{\nu,\nu+1}(\rr_+)\circ    x \subset M^{\ell_i-1,1}_{\nu,\nu+1}(\rr_+) + x  \circ M^{\ell_i,1}_{\nu,\nu+1}(\rr_+).
 \eeq
 Applying \eqref{eq:Mmmi} multiple times we conclude \eqref{eq:psic}.
  \end{proof}\medskip

 For all $u$ we define
 $$
 \gamma_+ u=2 \nu (x^{-\nu-\12}u)(x)|_{x=0}
 $$
 provided that the restriction to the boundary $\{x=0\}$ is well-defined, or more generally $\gamma_+ u =  2 \nu \fp_{x\to 0^+}(x^{-\nu-\12}u)(x)$ if the finite part is well-defined.

 \begin{proposition}\label{prop:key} Assume $\Re \nu>0$ and let $\ell \in \rr$ be such that  and $2\nu+\ell\notin \zz$.  Let  $B\in M^{\ell}_{\nu,\nu+1}(\rr_+)$. Then in the notation introduced in \eqref{eq:notI}, $\gamma_+ B \gamma_+^* \in  I^{\ell+2\nu}(Y;\Lambda_0,\Lambda_1)$. 
 \end{proposition} 
 \begin{proof}  Let us first consider the particular case when $B=\cH_\nu^{-1}b(\xi) \cH_\nu$ for some $b\in \cI^\ell$.
  Away from $\{x=x'=0\}$ the Schwartz kernel of $B$ is the function
 \beq\label{eq:bxxp}
 B(x,x')=\int_0^\infty  \cJ_\nu(x\xi)   b(\xi) \cJ_\nu(x'\xi)\,d\xi,
 \eeq
where for each $x,x'\in \open{0,+\infty}$ the integral converges absolutely thanks to the bounds \eqref{eq:bounds} on $\cJ_\nu$ and the condition  $\Re \alpha-2m >-1$ in the definition  \eqref{eq:defIl} of the $\cI^\ell$ classes.  

We will show that
 \beq\label{eq:gBgs}
 \gamma_+ B \gamma_+^*  = \gamma_+(\gamma_+ B^*)^*=  (2\nu)^2\fp_{x,x'\to 0^+,x>x'} x^{-2\nu-1} B(x,x')
 \eeq
 exists. To get a more practical expression for $B(x,x):=\lim_{x'\to x, x<x'}B(x,x')$ we use the following  Schläfli--Sonine  integral representation of the Bessel function  \cite[\S6.2]{Watson1922}:
 $$
 J_\nu(z)=\frac{ z^{\nu}}{2^{\nu+1} \pi i } \int_{c-\infty i}^{c+\infty i} t^{-\nu-1} e^{t-\frac{z^2}{4t}}dt
 $$
 valid in particular for all $\Re \nu>0$, $z>0$ and $c>0$. By inserting it into \eqref{eq:bxxp} we obtain
 $$
 B(x,x)= - \frac{x^{2\nu+1}}{2^{2\nu+2} \pi^2}\int_0^\infty  \int_{c-\infty i}^{c+\infty i}\int_{c-\infty i}^{c+\infty i} \xi^{2\nu+1} e^{s-\frac{x^2\xi^2}{4s}} e^{t-\frac{x^2\xi^2}{4t}} b(\xi)(st)^{-\nu-1} \,ds \,dt \,d\xi 
 $$
 
For the sake of notation brevity assume $m=1$, $\alpha= 1$, $s=0$; the same arguments will then apply in the general case with obvious adjustments in the numerology for powers of $\xi$ at $\infty$ and extra composition with $\xi$-independent differential operators on $Y$. 
Let us write $e^{-x^2(s+t)/4st}=:1+x\psi_{s,t}(x)$ and  $\xi^{2\nu+1}b(\xi)=\sum_{j=0}^N  c_j \xi^{2\nu-1-j}+r_N(\xi)$ for some $c_j\in \cD'(Y\times Y)$, $N\in \nn_0$ sufficiently large and $r_N\in \xi^{-q} L^1(\clopen{1,\infty};\cD'(Y\times Y))$ with $q>0$. Then 
\beq\label{eq:keypart}
\bea
\int_0^\infty e^{-x^2\xi^2\frac{s+t}{4st}} \xi^{2\nu+1} b(\xi) d\xi
&=\sum_{j=0}^N  c_j  \int_0^\infty e^{-x^2\xi^2\frac{s+t}{4st}} \xi^{2\nu-1-j} d\xi  \fantom + \int_0^\infty  r_N(\xi)d\xi+ x\int_0^\infty \xi \psi_{s,t}(x\xi) r_N(\xi)d\xi \\
&=    \frac{1}{2 }\sum_{j=0}^N  c_j   \textstyle\left(\frac{s+t}{4st}\right)^{-\nu+j/2}  \Gamma(\nu-\textstyle\frac{j}{2}) x^{-2\nu+j} \fantom + \int_0^\infty  r_N(\xi)d\xi+ x\int_0^\infty \xi \psi_{s,t}(x\xi) r_N(\xi)d\xi.
\eea
\eeq

Let now $r_{N,\varphi}(\xi)$ be any $\cD'(Y\times Y)$ semi-norm of $r_N(\xi)$. Then as $x\to 0^+$,
$$
\bea
x\int_0^\infty \xi \psi_{s,t}(x\xi) r_{N,\varphi}(\xi)d\xi &= O(x)  + \int_1^{\frac{1}{x}} \xi \psi_{s,t}(x\xi) r_{N,\varphi}(\xi)d\xi + \int_{\frac{1}{x}}^\infty  \xi \psi_{s,t}(x\xi) r_{N,\varphi}(\xi)d\xi,
\eea
$$ 
where we claim that the two integrals are $O(x^q)$ uniformly in $s,t\in c+ i\rr$. In fact, since $\psi_{s,t}(x\xi)$ is  uniformly bounded for $x\xi  \leq 1$, hence in particular when $1 \leq \xi  \leq \frac{1}{x}$, estimating the first integral amounts to showing the bound
$$
x\int_1^\frac{1}{x}\xi^{1-q} (\xi^q r_{N,\varphi}(\xi)) d\xi = O(x^{q})    
$$
which follows from the first part of \cite[Lemma 2.1.9]{lesch} (and a brief inspection of the proof gives uniformity in $s,t$). Similarly, from the definition of   $\psi_{s,t}(\xi)$  we have $\psi_{s,t}(\xi)=O(\xi^{-1})$ uniformly in $s,t$ as $\xi\to \infty$, hence $x\xi\psi_{s,t}(x\xi)$ is uniformly bounded when $\xi\geq \frac{1}{x}$, and estimating  the second integral amounts to 
$$
\int_{\frac{1}{x}}^\infty\xi^{-q} (\xi^q r_{N,\varphi}(\xi)) d\xi = O(x^{q}).    
$$ 
 which follows from the second part of \cite[Lemma 2.1.9]{lesch}. 
 
 Putting all this information together we conclude that 
 $$
 \bea
  \fp_{x\to 0^+} x^{-2\nu-1} B(x,x) &=- \frac{1}{2^{2\nu+2} \pi^2}\fint_0^\infty  \int_{c-\infty i}^{c+\infty i}\int_{c-\infty i}^{c+\infty i} \xi^{2\nu+1}  b(\xi) e^{s} e^{t} (st)^{-\nu-1} \,ds \,dt \,d\xi \\
 &=  \frac{(2\pi)^2}{2^{2\nu+2} \pi^2 \Gamma(\nu+1)^2}\fint_0^\infty  \xi^{2\nu+1}  b(\xi)  d\xi \\ &=  \frac{2^{-2\nu}}{\Gamma(\nu+1)^2}\fint_0^\infty  \xi^{2\nu+1}  b(\xi)  d\xi, 
 \eea
 $$
 where we computed the integrals in $s$ and $t$ by deforming the contour and by using Hankel's formula for the gamma function.  By combining this with \eqref{eq:gBgs} we obtain
 \beq\label{eq:gbgis}
 \bea
 \gamma_+ B \gamma_+^*  & =(2\nu)^2\fp_{x\to 0^+} x^{-2\nu-1} B(x,x) \\ &= (2\nu)^2  \frac{2^{-2\nu}}{\Gamma(\nu+1)^2}\fint_0^\infty  \xi^{2\nu+1}  b(\xi)  d\xi=  \frac{2^{-2\nu+2} }{\Gamma(\nu)^2} \fint_0^\infty  \xi^{2\nu+1} b(\xi) d\xi,
 \eea
 \eeq
 which belongs to  $I^{\ell+2\nu}(Y;\Lambda_0,\Lambda_1)$ by Proposition \ref{prop:fint}.
 
 Let us now consider the general case $B\in M^{m}_{\nu,\nu+1}(\rr_+)$. Then the expression \eqref{eq:bxxp} with $x'\to x$ is replaced by linear combinations of terms of the form
 \beq\label{eq:linkop}
 \int_0^\infty \cdots \int_0^\infty \cJ_{\nu_1}(x\xi_1) \prod_{i=1}^k   b_i(\xi_i) k_{\nu,i}(\xi_i,\xi_{i+1}) \cJ_{\nu_{k}}(x\xi_k)   d\xi_i 
 \eeq
 where  $\nu_1,\nu_k\in\{\nu, \nu+1\}$ and for $i=1,\dots,k-1$, $k_{\nu,i}(\xi,\eta)$ are linear combination of distributions of the form
  \beq\label{eq:rtt}
  \int_0^\infty \cJ_{\nu + j}(x\xi)\cJ_{\nu +l}(x\eta) dx, \quad j,l\in\{0, 1\},
 \eeq
  and  $k_{\nu,k}(\xi,\eta)\equiv1$. By representing $\cJ_{\nu_1}(x\xi_1)$ and $\cJ_{\nu_k}(x\xi_k)$ in \eqref{eq:linkop} through Schläfli--Sonine integrals as before, we are reduced to studying the $x\to 0^+$ behavior of expressions of the form
  $$
  \bea
  \int_0^\infty \cdots \int_0^\infty e^{\frac{-x^2\xi^2_1}{4s}} e^{\frac{-x^2\xi^2_k}{4t}}  \xi_1^{\nu_1+\frac{1}{2}} \xi_k^{\nu_k+\frac{1}{2}  }\prod_{i=1}^k   b_i(\xi_i) k_{\nu,i}(\xi_i,\xi_{i+1})  d\xi_i.  
 \eea
  $$
  In spherical coordinates $\xi= r \omega$, using   $k_{\nu,i}(\tau\xi,\tau\eta)=\tau^{-1}k_{\nu,i}(\xi,\eta)$ we obtain that this equals
    $$
    \bea
    \int_0^\infty  \int_{\mathbb{S}^{k-1}\cap \rr_+^k} e^{\frac{-x^2r^2 \omega_1^2}{4s}} e^{\frac{-x^2r^2 \omega_k^2}{4t}} r^{\nu_1+\nu_2+1}  \omega_1^{\nu_1+\12} \omega_k^{\nu_k+\12} \prod_{i=1}^k   b_i( r\omega_i) k_{\nu,i}(\omega_i,\omega_{i+1})   dr d\omega  
   \eea
    $$
    where $d\omega$ is the volume form on $\mathbb{S}^{k-1}$. 
    
    We can now repeat the arguments following \eqref{eq:keypart}, with $r$ playing the role of $\xi$, and with the modification that $c_j$ and $r_N$ depend additionally  on $\omega$ and there is extra integration in $\omega$ over the compact set $\mathbb{S}^{k-1}\cap \rr_+^k$. Since $c_j$ and $r_N$ are smooth in $\omega$, the extra integration  causes no complication. The resulting constant term in the $x\to 0^+$  expansion is a regularized  integral
    \beq\label{eq:rri}
    \bea {} & \fint_0^\infty  \int_{\mathbb{S}^{k-1}\cap \rr_+^k} r^{\nu_1+\nu_k+1}  \omega_1^{\nu_1+\12} \omega_k^{\nu_k+\12}\prod_{i=1}^k   b_i( r\omega_i) k_{\nu,i}(\omega_i,\omega_{i+1})   dr d\omega  \\
    &=   \fint_0^\infty   r^{\nu_1+\nu_k+1}  b(r)dr, \\
   \eea
    \eeq
    where
    $$
   b(r)= \int_{\mathbb{S}^{k-1}\cap \rr_+^k}  \omega_1^{\nu_1+\12} \omega_k^{\nu_k+\12}\prod_{i=1}^k   b_i( r\omega_i) k_{\nu,i}(\omega_i,\omega_{i+1})   dr d\omega.
   $$ 

While it is not expected in general that $b(r)\in \cI^\ell$, we claim that  the assertion of Proposition \ref{prop:fint} still holds true and   the regularized integral \eqref{eq:rri}  is in    $I^{\ell+2\nu}(Y;\Lambda_0,\Lambda_1)$.  In fact, we can repeat the proofs  starting from Lemma \ref{lem:bsm}  onwards for $\omega$-dependent  symbols, keep track of the dependence on $\omega$ and integrate.  In particular, disregarding low-frequency cutoffs $\varphi(\xi)$ in the notation,  the integral
 $$
 \int_1^\infty  \xi^{\alpha} \big(\bra \eta\ket^{p}\bra \sigma\ket^{l} + \xi^2\big)^{\frac{m}{2}}  d\xi 
 $$
is  replaced by the more general expression
 \beq\label{genexp}
  \int_{I}    \int_{\mathbb{S}^{k-1}\cap \rr_+^k} \prod_{i=1}^k r^{\Re (\nu_1 + \nu_k)+1 } \omega_1^{\nu_1+\12} \omega_k^{\nu_k+\12} \big(\bra \eta\ket^{p}\bra \sigma\ket^{l} + r^2\omega^2_i\big)^{\frac{m_i}{2}} | k_{\nu,i}(\xi_i,\xi_{i+1})|  d\xi_i, 
  \eeq
  where  $I$ is either $\clopen{1,\infty}$   or $\opencl{0,1}$. Using that $k_{\nu,i}(\tau\xi,\tau\eta)=\tau^{-1}k_{\nu,i}(\xi,\eta)$ for all $\tau>0$ and $i=1,\dots,k-1$ and denoting $\alpha=\Re (\nu_1+\nu_k)+1$,  by the change of variable that replaces $r^2$ by $\bra \eta \ket^p \bra \sigma \ket^l r^2$   we obtain that  \eqref{genexp} equals 
  $$
  \bea {}
 &( \bra\eta\ket^{p}\bra \sigma\ket^{l})^{\frac{|m|+\alpha+1}{2}}    \int_{\bra\eta\ket^{-p/2}\bra \sigma\ket^{-l/2} I}      \int_{\mathbb{S}^{k-1}\cap \rr_+^k} \prod_{i=1}^k r^{\alpha } \omega_1^{\nu_1+\12} \omega_k^{\nu_k+\12} \big(1+ r^2\omega^2_i\big)^{\frac{m_i}{2}} | k_{\nu,i}(\xi_i,\xi_{i+1})|  d\xi_i \\
 &\leq ( \bra\eta\ket^{p}\bra \sigma\ket^{l})^{\frac{|m|+\alpha+1}{2}}    \fint_0^\infty      \int_{\mathbb{S}^{k-1}\cap \rr_+^k} \prod_{i=1}^k r^{\alpha } \omega_1^{\nu_1+\12} \omega_k^{\nu_k+\12} \big(1+ r^2\omega^2_i\big)^{\frac{m_i}{2}} | k_{\nu,i}(\xi_i,\xi_{i+1})|  d\xi_i  \\  
 & =  C_{\nu,m}  \bra\eta\ket^{\frac{p(|m|+\alpha+1)}{2}}  \bra \sigma\ket^{\frac{l(|m|+\alpha+1)}{2}}.    
\eea
  $$
  The remaining steps can be then repeated with  obvious notational adjustments.
 \end{proof}
 
 \medskip
 
 In the next lemma,  $H_{\nu,\b,+}^{1,m}(\rr_+\times Y)$ is the future-supported version of $H_{\nu,\b,\loc}^{1,m}(\rr_+\times Y)$.
 
  \begin{lemma}\label{prop:key2} Assume $\Re \nu>0$  and let $\ell\leq 0$ be such that  $\nu+\12+\ell\notin \zz$.  If $B\in M^{\ell}_{\nu,\nu+1}(\rr_+)$  then 
  $$\gamma_+ B: H_{\nu,\b,+}^{0,m}(\rr_+\times Y)\to   H_{\loc}^{m-\ell-\nu+1}(Y)$$ continuously for all $m\in \rr$.
  \end{lemma}
  \begin{proof} For the sake of notational simplicity we  focus on the case when $B= \cH_\nu^{-1} b(\xi) \cH_\nu$ for some $b\in \cI^\ell$. Let $s=-\ell - \nu+3/2$. Since $D_y$ commutes with $\cH_\nu$ and $\cH_\nu\in B(L^2(\rr_+\times Y))$, 
  $$\cH_\nu :  H_{\nu,\b,+}^{0,m}(\rr_+\times Y)\to  \bra D_y\ket^{-m} L^2(\rr_+\times Y)
  $$ continuously, so it suffices to prove $\gamma_+ \cH_\nu^{-1} b(\xi) \bra D_y \ket^{-m}:L^2_{+}(\rr_+\times Y)\to H_{\loc}^{m+s}(Y)$ continuously or equivalently 
  \beq\label{eaea}
  \gamma_+ \cH_\nu^{-1} \bra D_y \ket^{m+s} b(\xi) \bra D_y \ket^{-m}:L^2_{+}(\rr_+\times Y)\to L^2_{\loc}(\rr_+\times Y).
  \eeq
  Dropping the `ft' and `loc' subscripts and disregarding localization functions for the sake of brevity, it suffices to prove an estimate on simple tensors $\varphi\otimes v \in L^2(\rr_+\times Y)$.  Proceeding as in the proof of  Proposition \ref{prop:key}, we obtain that for $\varphi\otimes v\in L^2(\rr_+\times Y)$,
  $$
  \bea
 &\gamma_+ \cH_\nu^{-1} \bra D_y \ket^{m+s} b(\xi) \bra D_y \ket^{-m}(\varphi\otimes v) \\ &= \frac{2^{-2\nu}}{\Gamma(\nu)}\left(\fint_0^\infty \xi^{\nu +\frac{1}{2}}\bra D_y \ket^{m+s}  b(\xi) \bra D_y \ket^{-m} \varphi(\xi) d\xi \right)v =: A v.
 \eea
  $$
By Proposition \ref{prop:fint}, $A\in I^{0}(Y;\Lambda_0,\Lambda_1)$, with seminorms bounded by $\norm{\varphi}_{L^2(\rr_+)}$. This then implies \eqref{eaea} by mapping properties of paired Lagrangian distributions.  The  case of general $B\in M^{\ell}_{\nu,\nu+1}(\rr_+)$ is proved by modifying the previous arguments analogously to the second part of the proof of Proposition \ref{prop:key}. 
  \end{proof}

 \section{Symbolic approximation of Dirichlet-to-Neumann map} \label{sec:DN}
 
 \subsection{Parametrix through Hankel multipliers}  Let us recall that on an asymptotically AdS spacetime $(X,g)$, the Klein--Gordon operator, rescaled as in \sec{s:preliminaries} by suitable powers of $x$, is given by
 \beq\label{eq:req}
 \bea
 P&= D_x^2 + \big(\nu^2-{\textstyle\frac{1}{4}}\big){x^{-2}}+c(x)\big(x\p_x +\textstyle \frac{d-1}{2}\big)+ \square_{h(x)} \\  &= N_{\nu} + c(x)\big(x\p_x +\textstyle \frac{d-1}{2}\big) +\square_{h(x)}
 \eea
 \eeq
in a collar neighborhood $\clopen{0,\varepsilon}\times \pX$ of the boundary.

 \begin{theorem}\label{thm:DN} Suppose $\nu>0$ (or $\Re\nu>0$ in the case when $\hat g$ is a product metric) and $\nu\notin\12\nn_0$. The (generalized) forward Dirichlet-to-Neumann map $\DN_{g}(\nu)$ of the Klein--Gordon operator $\square_g + \nu^2 - \tfrac{d^2}{4}$  is a paired Lagrangian distribution in   $I^{\nu-\12,\nu+\12}(\pX\times \pX; {\Lambda}_0,{\Lambda}_1)$ which satisfies
 \beq\label{eq:toprove1}
\Lambda_{g}(\nu)= \cons (\square_{h(0)})^{\nu} \mod I^{\nu-\frac{3}{2},\nu-\12}(\pX\times \pX; {\Lambda}_0,{\Lambda}_1).
 \eeq
 \end{theorem} 
 
 \begin{proof} The proof is divided into several steps. As previously, the forward Dirichlet  propagator of $P$ is  abbreviated by $P^{-1}$.

 \step{1} In the first step we attempt to construct a parametrix for $P^{-1}$ close to the boundary at leading order in $x$. Let 
 $$
 P_0=     D_x^2 + \big(\nu^2-{\textstyle\frac{1}{4}}\big){x^{-2}}+ \square_{h(0)},
 $$   
as a differential operator on  $\rr_+\times \p X$. Then  $P-P_0= c(x) (x\p_x + \frac{d-1}{2})  + \square_{h(
x)}- \square_{h(0)}$ close to the boundary. Let $\chi_1\in \cf(\rr_+)$ be such that $\chi_1\equiv 1$ for $x\leq \frac{\varepsilon}{4}$ and $\chi_1\equiv 0$ for $x\geq \frac{3\varepsilon}{8}$. For $N\in \nn$, by Taylor expanding in $x$ to order $N-1$ the function $c(x)$  and to order $N$ the coefficients of $\square_{h(x)}$, we obtain  
$$
\chi_1 P= P_0 + \sum_{k=1}^{N+1}  P_k, \quad P_k\defeq  x^{k-1} B_{k-1}   +  x^k L_k, 
$$
where $B_{k}=\frac{1}{k!} c^{(k)}(0) (x\p_x + \frac{d-1}{2}) \in  \Diff_\b^1(\rr_+)$, $L_k=\frac{1}{k!} (\p^k\,\square_{h(x)}/\p x^k)|_{x=0}  \in \Diff^2(Y)$ for $k=1,\dots,K$, and   $P_{N+1}\in \Diff_\b^1(\rr_+; \Diff^2(Y))$ accounts for the Taylor remainder and for terms supported away from $x=0$. 

Let $\chi\in C_{\rm c}^\infty(\rr_+)$ be such that $\chi\equiv 1$ in a neighborhood of $\xi=0$ and let $\chi(N_\nu)=\cH_\nu^{-1} \chi(\xi) \cH_\nu$. Let  
$
Q= \sum_{l=0}^{N+1} Q_{-l}
$, 
where $Q_{-l}$ is defined by the recursion
\beq\label{eq:recursion}
\begin{cases}
Q_{0}= P_0^{-1}(\one-\chi(N_\nu)), \\
Q_{-l}=  - \sum_{k=1}^l  Q_{-l+k}    P_k  P_0^{-1}(\one-\chi(N_\nu)), \ \ l=1,\dots, N,
\end{cases}
\eeq
with $P_0^{-1}$ being the forward Dirichlet propagator of $P_0$. Then,
\beq\label{eq:recursion2}
\bea
Q\chi_1 P&=  \sum_{j=0}^{N+1}\sum_{k=0}^{N+1}  Q_{-j}  P_k = \sum_{l=0}^{N+1}\sum_{k=0}^l    Q_{-l+k}  P_k   - R_{-N-2} \\
&= Q_0 P_0  +\sum_{l=1}^{N+1} \bigg( Q_{-l} P_0 -  \sum_{k=1}^l    Q_{-l+k}  P_k \bigg) -R_{-N-2} \\
&= \one - R_{-N-2}-\chi(N_\nu).
\eea
\eeq
where we have denoted
$$
 R_{-N-2}=-\sum_{l=N+2}^{2N+2}\sum_{k=N+1}^{l-N-1}    Q_{-l+k}  P_k.
$$
Let $\chi_2(x)=\chi_1(x/2)$. Then $\chi_1 P \chi_2= \chi_1 P$, thus by multiplying both sites of \eqref{eq:recursion2} by $\chi_2$ from the right we obtain
$$
Q \chi_1 P = \chi_2 -  (R_{-N-2}+\chi(N_\nu)) \chi_2.
$$
Thus, by uniqueness of the forward Dirichlet inverse we have $$
\chi_2P^{-1}  =Q \chi_1 + (R_{-N-2}+\chi(N_\nu)) \chi_2 P^{-1},$$ 
hence by Corollary \ref{cor:DN},  
$$
\bea
\Lambda_{g}(\nu) &= (2\nu)^{-1} \gamma_+ P^{-1} \gamma_+^*=  (2\nu)^{-1}\gamma_+ P^{-1} \chi_2  \gamma_+^*\\ &=  (2\nu)^{-1} \bigg(\sum_{j=0}^{N+1} \gamma_+ Q_{-j} \gamma^*_+  +  \gamma_+ (R_{-N-2}+\chi(N_\nu)) \chi_2 P^{-1} \gamma_+^*\bigg).
\eea
$$

Our aim now is  to show that the $Q_{-j}$ and $R_{-N-2}$ contributions get more regular for higher $j$ and $N$ respectively. In the recursive formula \eqref{eq:recursion} for $Q_{-j}$, we first commute to the left all  the $x\p_x$ terms (arising from the $P_k$), and then within the right factor we  commute all remaining powers of $x$ to the left. In the  first step (commuting polynomials in $x\p_x$) we use the operator identity $x^\lambda  x \p_x= (x \p_x-\lambda) x^\lambda$ and
$$
P_0^{-1} x\p_x = (x\p_x + 2 )P_0^{-1} - 2  P_0^{-1}  \square_{h(0)} P_0^{-1}, 
$$ 
which follows from the identity  $x \p_x P_0 = P_0 (x\p_x +2) - 2 \square_{h(0)}$. In the second step (commuting powers of $x$) we use that $P_0^{-1}\in M_{\nu,\nu+1}^{-2,1}(\rr_+)$ and $\square_{h(0)}\in M_{\nu,\nu+1}^{0,1}(\rr_+)$ and we apply Lemma  \ref{lem:NN} repeatedly.    As a result we conclude 
$$
Q_{-j}\in ( \Diff^j_\b(\rr_+) \otimes \one ) (M^{-2-j}_{\nu,\nu+1}(\rr_+) + x M^{-2}_{\nu,\nu+1}(\rr_+))
$$  
and similarly $R_{-j}\in ( \Diff^j_\b(\rr_+) \otimes \one ) (M^{-j}_{\nu,\nu+1}(\rr_+) + x M^{0}_{\nu,\nu+1})(\rr_+)$. 

Since $\gamma_+\circ  x \p_x=  (\nu +\12 ) \gamma_+$, the $\Diff_\b(\rr_+)$ factor contributes only with a non-zero multiplicative constant in the computation of  $\gamma_+ Q_{-j} \gamma_+^*$  and $\gamma_+ R_{-j}\chi_2 P^{-1}\gamma_+^*$ (and the $O(x)$ term does not contribute at all). In consequence,
$$
\gamma_+  Q_{-j} \gamma_+^* \in \gamma_+ M^{-2-j}_{\nu,\nu+1}(\rr_+) \gamma_+^* , \quad  \gamma_+ R_{-j} \chi_2 P^{-1}\gamma_+^* \in \gamma_+ M^{-j}_{\nu,\nu+1}(\rr_+) P^{-1} \gamma_+^*.
$$
By Proposition \ref{prop:key}, $\gamma_+  Q_{-j} \gamma_+^* \in  I^{\nu-\12-j,\nu+\12-j}(\pX\times \pX;\Lambda_0,\Lambda_1)$. Furthermore,
$$
\gamma_+ R_{-j} \chi_2 P^{-1}\gamma_+^*:  H_{\rm c}^{m}(\pX) \to H_{\rm loc}^{m-2\nu+j}(\pX)
$$
continuously, as follows from mapping properties of $P^{-1}\gamma_+^*$ (see  the proof of Proposition \ref{cor:DN}) and of $\gamma_+ R_{-j}$ (see Lemma \ref{prop:key2}). Similarly, $\gamma_+ \chi(N_\nu) \chi_2 P^{-1}\gamma_+^*$ is regularizing; the difference is that we cannot deduce mapping properties of $\gamma_+ \chi(N_\nu)$ directly from Lemma \ref{prop:key2} ($\chi(N_\nu)$ is not in $\cI^\ell$ because $\chi$ is not supported away from $\xi=0$), but all the arguments of the proof apply as everything is well behaved near $\xi=0$.   In summary, we obtain that
$$
\Lambda_{g}(\nu) = \sum_{j=0}^{N+1} \Lambda_{g,j}(\nu) \mod B\big(H_{\rm c}^{m}(\pX) , H_{\rm loc}^{m-2\nu+ j}(\pX)\big),
$$
where $\Lambda_{g,j}(\nu)=(2\nu)^{-1}\gamma_+ Q_{-j} \gamma_+^*\in  I^{\nu-\12-j,\nu+\12-j}(\pX\times \pX;\Lambda_0,\Lambda_1)$. In particular,
$$
\Lambda_{g,0}(\nu)= (2\nu)^{-1}\gamma_+  P_0^{-1} \gamma_+^*=  \cons (\square_{h(0)})^\nu
$$ by \eqref{eq:gbgis} and Definition \ref{def:cp}  (cf.~the formal computation in \sec{ss:productcase}). This proves \eqref{eq:toprove1}.\end{proof}

\subsection{Differences of Dirichlet-to-Neumann maps}

We now show that as in the asymptotically hyperbolic setting \cite{Joshi2000}, for $\nu$ outside of a discrete subset of parameters, $\Lambda_{g}(\nu)$ determines the Taylor series of the metric at $\pX=\{x=0\}$.

\begin{theorem}\label{T.Taylor}
Let $\nu > 0$. There exists a countable set $Z\subset \rr$ containing $\12\nn_0$ and such that for $\nu \notin Z$ the following holds true. For $j=1,2$, let $(X_j,g_j)$ be an  asymptotically anti-de Sitter spacetime satisfying Hypothesis~\ref{hyp:global}, and write  $g_j=x_j^{-2}(-dx_j^2+h_j)$ as in~\eqref{eq:fform}. Suppose that $Y:=\partial X_1=\partial X_2$, and the corresponding generalized Dirichlet-to-Neumann maps are equal: $\Lambda_{g_1}(\nu)=\Lambda_{g_2}(\nu)$. Then the Taylor series of the metric at the boundary coincides, that is,
	\begin{equation}\label{eq:Taylor}
	\partial_{x_1}^kh_1|_{x_1=0}=\partial_{x_2}^kh_2|_{x_2=0}\,, \ \  k\geq0\,.
	\end{equation}
	\end{theorem}

\begin{proof}
Suppose for  that $h_1(x)$ and  $h_2(x)$ are two metrics such that for some $N\in\nn$, $h_1^{(k)}(0)=h_2^{(k)}(0)$ for $k=1,\dots,N-1$ and $h_1^{(N)}(0)\neq h_2^{(N)}(0)$.   Let us denote the corresponding asymptotically AdS metrics, differential operators and related objects entering the proof of Theorem \ref{thm:DN} with an extra subscript $i=1,2$. We will show that  $\Lambda_{g_1,j}(\nu)=\Lambda_{g_2,j}(\nu)$ for $j=1,\dots,N-1$ but  $\Lambda_{g_1,N}(\nu)\neq\Lambda_{g_2,N}(\nu)$. 

We have  $B_{1,k-1}=B_{2,k-1}$ for $k=1,\dots,N-1$   and $B_{1,N-1}\neq B_{2,N-1}$  (note that $c(x)= \12 x^{-1}\partial_x (\log |{\det h(x)}|)$ involves  one derivative of $h(x)$ already). Similarly, $L_{1,k}=L_{2,k}$ for $k=1,\dots,N-1$  and $L_{1,N}\neq L_{2,N}$, hence $P_{1,k}=P_{2,k}$ for $k=1,\dots,N-1$  and $P_{1,N}\neq P_{2,N}$.
 From the definition  \eqref{eq:recursion} of $Q_{i,-l}$ we conclude that $Q_{1,-l}=Q_{2,-l}$ for $l=1,\dots,N-1$, and
 $$
 \bea
 Q_{1,-N}-Q_{2,-N} &= Q_{0} (P_{2,N}-P_{1,N})P_0^{-1}\\
 &= P_0^{-1} x^{N} (L_{2,N}- L_{1,N}) P_0^{-1} + P_0^{-1} x^{N-1} (B_{2,N}- B_{1,N}) P_0^{-1}.\\
 &=  P_0^{-1} x^{N} P_0^{-1} (L_{2,N}- L_{1,N}) + P_0^{-1} x^{N}  P_0^{-1}[ \square_{h(0)}, L_{2,N}- L_{1,N}] P_0^{-1} \fantom + P_0^{-1} x^{N-1} (B_{2,N}- B_{1,N}) P_0^{-1}
 \eea
 $$
 Here the second and third summand contribute with a lower order term to the difference  $\Lambda_{g_1}(\nu)-\Lambda_{g_2}(\nu)$, in comparison to the first summand. Since  $L_{2,N}- L_{1,N}\neq 0$ commutes with $\gamma_+$, it suffices to inspect the expression 
 $$
E_N\defeq \gamma_+  P_0^{-1} x^{N}  P_0^{-1} \gamma_+^*.
 $$
As before, we commute powers of $x$ to the left. Concretely, since $P_0^{-1}= \cH_\nu^{-1}  p_0^{-1}(\xi)\cH_\nu$ with  $p_0^{-1}(\xi)=(\xi^2+ \square_{h(0)})^{-1}$,  iterated use of the identities \eqref{eq:hankelx} gives 
$$  
E_N =
  \gamma_+  \cH_\nu^{-1}  \big( ( - (\p_\xi + (2\nu+1) \xi^{-1})\p_\xi  )^{K}p_0^{-1}(\xi)\big)p_0^{-1}(\xi)   \cH_\nu    \gamma_+^* 
 $$ 
for $N=2K$ and
$$  
E_N =
   \gamma_+  \cH_\nu^{-1}  \big( -\p_\xi( - (\p_\xi + (2\nu+1) \xi^{-1})\p_\xi  )^{K} p_0^{-1}(\xi)\big) \cH_{\nu+1} \cH_\nu^{-1} p_0^{-1}(\xi)   \cH_\nu   \gamma_+^*, 
$$
for $N=2K+1$.
The principal symbol of $E_N$ can be computed  following the steps in the proof of Proposition \ref{prop:key} and it can vanish only if a certain polynomial of $\nu$ vanishes, which happens at most for $\nu$ in a discrete set (note that the $\nu$-dependent constants are the same as  in the Riemannian analogue of the problem, in which case precise computations were given in  \cite{Joshi2000}).
\end{proof}

\medskip

	A first direct corollary for the bulk metrics holds true in the analytic case.
	
 \begin{corollary}\label{T.analytic} 
 	Under the hypotheses of Theorem~\ref{T.Taylor}, if $(X_i, g_i)$ and~$x_i$ are real analytic and $\p X_i$ is simply connected for $i=1,2$, then $(X_1,g_1)$ and $(X_2,g_2)$ are isometric. The same assertion holds true if $(X_i,g_i)$ solves Einstein equations with negative cosmological constant or more generally, satisfies the averaged null condition.     
 \end{corollary}
 \begin{proof}
  	Equation~\eqref{eq:Taylor} ensures that there are neighborhoods of the boundary $X_i'\subset X_i$ with analytic charts $\Phi_i:\clopen{0,\epsilon}\times Y\to X_i'$  for some $\epsilon>0$, such that $\Phi_1^{-1}\circ\Phi_2:(X_1',g_1)\to (X_2',g_2)$ is an analytic isometry. If $X_i$ is simply connected then by analyticity we can use the  arguments of Lee--Uhlmann \cite{Lee1989}  to  show that this extends to a global isometry $(X_1,g_1)\to (X_2,g_2)$. This requires the topological assumption that $\pi_1(X_i,\p X_i)=0$, i.e.~every closed path in $X_i$ with base point in $\p X_i$ is homotopic to a path that lies entirely in $\p X_i$. This is automatically verified if the boundary is simply connected.   In the second variant it follows  from a theorem of Galloway--Schleich--Witt--Woolgar \cite{galloway} using our global hyperbolicity assumption. 
  \end{proof}
 
 \begin{remark}\label{T.variant} Another variant can be obtained if instead of the assumption that $\pX_i$ is simply connected, $(X_i,g_i)$ is foliated by complete analytic Cauchy surfaces of dimension $d\geq 3$. 
 
 In this situation we can argue as follows. Using \cite{sanchez} we can assume without loss of generality that $\hat g_i=\beta_i d\tau^2 - k_{i}(\tau)$, with $\beta$ real analytic and $k_{i}(\tau)$ a family of real analytic Riemannian metrics. By analytic continuation, $\beta_1=\beta_2$. Next,  choose  neighborhoods $X_i''$ of the boundary strictly contained in $X_i'$, $i=1,2$. There exists neighborhoods of $\p X_1''\simeq \p X_2''$ in $X_1$ and $X_2$ respectively that are isometric, in particular the Cauchy surface $(\Sigma_{1,\tau}, k_1(\tau))$ is isometric to $(\Sigma_{2,\tau}, k_2(\tau))$ in a neighborhood of $\p \Sigma_{1,\tau}\simeq \p \Sigma_{2,\tau}$ for each $\tau\in \rr$. Therefore for each $\tau \in \rr$ the Riemannian Dirichlet-to-Neumann maps associated to the metric $k_i(\tau)$ at $\p X_i''$ are equal. By a result of Lassas--Taylor--Uhlmann \cite{Lassas2003} we conclude that there is a global isometry $(\Sigma_{1,\tau}, k_1(\tau))\to (\Sigma_{2,\tau}, k_2(\tau))$. Furthermore, the isometry depends smoothly on~$\tau$ because, in the construction of~\cite{Lassas2003}, $\tau$ only appears through the Green's function $G_{i,\tau}$ of the Laplacian of $(\Sigma_{i,\tau}, k_i(\tau))$, which itself a smooth function of~$\tau$ (with values in a suitable Sobolev space of negative index).
 \end{remark}

\section{Einstein asymptotically anti-de Sitter metrics} \label{sec:einstein}
 
\subsection{Inverse problem for Einstein asumptotically AdS metrics}
  In this final section we consider asymptotically anti-de Sitter Einstein metrics, i.e., which satisfy
$$
  \mathrm{Ric}_g+(n-1) g=0.
  $$
  Note that if the boundary dimension $d=n-1$ is odd, Einstein metrics are typically not smooth up to the boundary due to the presence of logarithmic terms --- smooth Einstein metrics exist in all dimensions, though. For the sake of simplicity we  restrict ourselves to the smooth case.
  
  In this setting, an inverse result can be obtained by combining Theorem \ref{T.Taylor} with a unique continuation theorem due to Holzegel--Shao \cite{Holzegel2023}. We first recall the main assumption used in the latter.
  
  \begin{definition}[{Chatzikaleas--Shao \cite{CS22}}] \label{nullco} An asymptotically anti-de Sitter spacetime $(X,g)$  satisfies the \emph{generalized null convexity criterion} if there exists $\eta\in C^4(X)$ such that $\eta>0$ on $X$, $\eta=0$ on $\p X$ and the bilinear form 
  $$
  \eta^{-1} D^2_g \eta + P_g
  $$
 is uniformly positive definite on all null covectors, where $D^2_g$ is the Hessian and $P_g=\frac{1}{n-2}({\rm Ric}_g - \frac{1}{2(n-1)}R_g g)$ the Schouten tensor of $g$.
  \end{definition}
   
  \begin{theorem}\label{E.Einstein}  
  For $i=1,2$,	suppose that $(X_i,g_i)$ is an Einstein asymptotically anti-de Sitter spacetimes satisfying the generalized null convexity criterion in a neighborhood of~$Y$ in~$X_i$. Then, under the hypotheses of Theorem~\ref{T.Taylor}, there exists a neighborhood $X_i'$ of~$Y$ in~$X_i$ such that $(X_1',g_1)$ and $(X_2',g_2)$ are isometric. 
  \end{theorem}
  
  \begin{proof}
By Theorem \ref{T.Taylor}, there exists a diffeomorphism $\Phi_j: \clopen{0,\epsilon}\times Y\to X_j'$ in $X_j$ such that the pulled back metrics $\Phi_j^*g_j$ have the same Taylor series on $Y$. We can choose $\epsilon$ small enough that $X_j'$ is contained in the neighborhood $X_j''$ where the generalized null convexity criterion is known to hold. By the  unique continuation result of Holzegel--Shao \cite[Thm.~1.5]{Holzegel2023}, the metric $\Phi_j^*g_j|_{X_j'}$ is determined by its $(n-1)$-th order Taylor series on~$Y$, so we conclude that $\Phi_1^{-1}\circ\Phi_2$ is a diffeomorphism $(X_1',g_1)\to (X_2,g_2)$.

  \end{proof}
  
 \subsection{Dirichlet-to-Neumann map and conformal operators} Finally, again in the setting of smooth Einstein asymptotically anti-de Sitter spacetimes, we show an extension of the Graham--Zworski theorem to Lorentzian signature.
 
 We first state  a  direct analogue of \cite[Prop.~4.2]{grahamzworski}, the proof of which applies verbatim, with forward inverses replacing the resolvent.
 
\begin{proposition}\label{prop:gz2} Let $(X,g)$ be an Einstein asymptotically anti-de Sitter spacetime of dimension $n=1+d$, and let $f\in C^\infty_+(\p X)$. If $\nu=k\in \nn$, there is a formal solution of $P u = O(x^\infty)$ of the form 
$$
u = x^{k-\12}(F+ G x^{2k} \log x), \quad F,G\in C^\infty(X)
$$
with $F|_{\pX}= f$. Here $F$, resp.~$G$, is invariantly determined modulo $O(x^{2k})$, resp.~$O(x^{\infty})$. Moreover, 
$$
G|_{\pX} = - 2 c_k L_k f, \quad  c_k=(-1)^{k} \frac{1}{2^{2k} k! (k-1)!}, 
$$
where $L_k\in \Diff^{2k}(\pX)$ has principal part $\square_{h(0)}^k$. If $d\in 2\nn_0 +1$ or if $d\in 2\nn$ and $k\leq \frac{d}{2}$, then $L_k$  depends only on $h(0)$ and defines a conformally invariant operator.
\end{proposition} 
 
In combination with our analysis this implies the following Graham--Zworski type theorem \cite{grahamzworski}. 
 
 \begin{theorem}\label{thm:gza}  Let $(X,g)$ be an Einstein asymptotically anti-de Sitter spacetime of dimension $n=1+d$ satisfying Hypothesis~\ref{hyp:global}. Let $k\in \nn$ and assume $k\leq \frac{d}{2}$ if $d$ is even. Then  
  \beq\label{eq:resi2}
   L_k  = (-1)^{k+1} 2^{2k} k!(k-1)!   \big( \lim_{\nu\to k} (\nu-k) \Lambda_g(\nu)\big), 
  \eeq
 is a conformally invariant differential operator in $\Diff^{2k}(\pX)$ with principal part $\square_{h(0)}^k$. Furthermore, the limit in \eqref{eq:resi} equals the residue of a meromorphic family of paired Lagrangian distributions that coincide with $\Lambda_g(\nu)$ for $\nu>0$.
 \end{theorem}
 \begin{proof} At this point the arguments from the proof of \cite[Thm.~1]{grahamzworski} apply to our setting with only minor alterations, namely, with Propositions 3.5 and 4.2 therein  replaced by  Propositions \ref{prop:gz1} and \ref{prop:gz2} respectively to yield \eqref{eq:resi2}. Note that the key Proposition 3.6 in \cite{grahamzworski} which computes the residue of $\Lambda_g(\nu)$  simplifies in our setting as there are no contributions from eigenvectors. Next, the parametrix  $\Lambda_g^{(N)}(\nu):=\sum_{j=0}^{N+1} \Lambda_{g,j}(\nu)$ constructed in the proof of  Theorem \ref{thm:DN} is well defined as a meromorphic function of $\nu$ as  its definition refers to the family of forward inverses in the product case  (see Theorem \ref{thm:Pinv}). For  $\nu$ real  the only contribution to the limit in \eqref{eq:resi2} comes from   $\Lambda_g^{(N)}(\nu)$---in fact from the product case term $\Lambda_{g,0}(\nu)$. Thus the same is true for the pole of the analytic continuation of  $\Lambda_g(\nu)$ by real analyticity.
 \end{proof} 
 
 \appendix
 
\section{Energy estimates for complex $\nu$} \label{sec:app}
 
 \subsection{Preliminaries} Let $\Re \nu >0$. In coordinates $$
 (z^0,z^1,\ldots, z^{n-1})=(x,y^0,\ldots,y^{n-2})$$
 near $\pX$  as in   \sec{ss:AdS} we introduce the following notation for twisted derivatives:
 $$
 Q_i= x^{-\nu+\12}  \p_{z_i}  x^{\nu-\12}\in \Diff_\nu^1(X), \ \ i=0,\ldots, n-1,
 $$
 so that $Q_0$ depends on $\nu$ but $Q_i$ are just usual derivatives. Note that in the complex case,  $N_\nu= - \p_x^2 + \big(\nu^2-{\textstyle\frac{1}{4}}\big){x^{-2}}={\bar Q_0^*} { Q_0}$; the complex conjugate serves to replace $\nu$ by {$\bar\nu$} so that the expression is analytic in $\nu$.

 The   twisted Sobolev space $H^1_\nu(X)$ norm can be written as
 $$
 \| u \|_{H^1_\nu}= \| u \|_{L^2} +  \bra \delta^{ij} Q_i u, Q_j u\ket_{L^2}, 
 $$
  where repeated indexes are summed from $0$ to $d=n-1$ and we stress again that the index zero corresponds to the $x$ variable (rather than $t$ as common in relativity). 
  
When $\nu\notin\rr$, the natural replacement of the Dirichlet form of $P=P(\nu)$ is
$$
  \mathcal{E}_0(u,v)= \langle P_\dir u ,v \rangle_{L^2} = -\langle \hat g^{ij} Q_i u , {\bar Q_j} v \rangle_{L^2}.
   $$
 It is however ill-suited for energy estimates, the main reason being that it does not appear to be related in any useful way to a positive-definite quantity such as $\| u \|_{H^1_\nu}$ since $\bar Q_0\neq Q_0$.

This motivates us to consider a modification involving non-differential expressions. Recall that the Hankel transform $\cH_\nu$ was introduced in \S\ref{ss:hankel} for $\Re \nu >-1$ and satisfies $\cH_\nu^2=\one$ and $\cH_\nu^*=\cH_{\bar\nu}$.  Furthermore, $Q_0=\cH_{\nu-1}^{-1} \xi \cH_\nu$. Let 
$$
\cI_\nu:=\cH_{\bar \nu}^{-1} \cH_\nu= \cH_\nu^* \cH_\nu
 $$
 understood as an operator acting on $L^2(X)$ functions supported close to $\pX$.  We introduce a modified Dirichlet form
$$
2  \mathcal{\tilde E}_0(u,v):=  -\langle \hat g^{ij} Q_i u , {\bar Q_j} \cI_\nu v \rangle_{L^2}-\langle \hat g^{ij} {\bar Q_i} \cI_\nu u ,  Q_j  v \rangle_{L^2}.
   $$ 
 The main motivation is that
\beq\label{eq:idqq0}
\bar Q_0 \cI_{\nu}  = \cI_{\nu-1} Q_0,  \quad \bar Q_j \cI_{\nu}  = Q_j \cI_\nu= \cI_{\nu} Q_j,  \  \ j =1,\dots,d,
 \eeq
 hence
  \beq\label{eq:idqq}
 Q_0^* \bar Q_0 \cI_{\nu}  =  Q_0^* \cI_{\nu-1} Q_0\geq 0, \quad Q_j^* \bar Q_j \cI_{\nu}  =  Q_j^* \cI_{\nu} Q_j\geq 0, \  \ j =0,\dots,d,
  \eeq
Furthermore,
$$
Q_i^* \bar Q_j \cI_{\nu}= \cI_{\nu} \bar Q_i^*  Q_j \mbox{ if } i=j=0 \mbox{ or } i,j\neq 0.    
$$

\begin{lemma} Let
$$
\| u \|_{\tilde H^1_\nu}:= \| u \|_{L^2} +  \bra \delta^{ij} Q_i u, \bar Q_j  \cI_{\nu} u \ket_{L^2}.
$$
 Then the norm $\| \cdot \|_{\tilde H^1_\nu}$ is equivalent to the twisted Sobolev norm $\| \cdot \|_{H^1_\nu}$.
\end{lemma} 
\begin{proof} This follows easily from \eqref{eq:idqq} combined with boundedness and invertibility of $\cI_{\nu-1}$ and $\cI_\nu$.
\end{proof}

\begin{lemma}\label{lem:com} Let $f\in C^1(X)\cap L^\infty(X)$ and denote in the same way the corresponding multiplication operator. Then for $\Re \nu >-1$, $[\cI_\nu, f] x\p_x \in B(L^2(X))$ and
\beq
\bea {}
\| [\cI_\nu, f]\|_{B(L^2(X))} & \leq C \| x f'\|_{L^\infty}, \\   \|[\cI_\nu, f] x\p_x\|_{B(L^2(X))} &\leq C (\| x f'\|_{L^\infty}  +  \| x^2 f''\|_{L^\infty}).
\eea
 \eeq
 The same conclusion holds true if $\cI_\nu$ is replaced by $\cI_\nu^{\12}$. 
\end{lemma}
\begin{proof} By  \cite[Prop.~4.5]{Derezinski2017} (cf.~\cite[Thm.~6.2]{Bruneau2011}), $\cI_\nu=g_\nu(D)$, where $D= \12 (xD_x + D_x x)=i^{-1}(x\p_x +\12)$ and $g_\nu$ is the real function
$$
g_\nu(s)=\frac{\Gamma\left(\frac{\nu+1-i s}{2}\right)\Gamma\left(\frac{\bar{\nu}+1+i s}{2}\right)}{\Gamma\left(\frac{\nu+1+i s}{2}\right)\Gamma\left(\frac{\bar{\nu}+1-i s}{2}\right)}.
$$ 
Using well-known asymptotics of the gamma function and digamma function we can show that $g_\nu\in S^0(\rr)$. By direct inspection the iterated commutators satisfy 
\beq\label{adD}
\| {\rm ad}^j_D f \|_{B(L^2(X))}\leq C (\| x f'\|_{L^\infty}  +  \| x^2 f''\|_{L^\infty}), \ \ j=1,2,
\eeq
in particular ${\rm ad}^j_D f $ is bounded if the RHS is finite. Supposing it is, we can apply the commutator expansion \cite[Lem.~C.3.1]{DGscattering} (strictly speaking its direct analogue where the order of operators is reversed, which follows from the same arguments) with $\rho=0$, $N=0$ and  then $N=1$, which gives
$$
[g_\nu(D),f]= [D,f] g'_\nu(D) + R(f,g_\nu,D)
$$
where $\|  R(f,g_\nu,D)D^{j} \|_{B(L^2(X))}\leq C \| {\rm ad}^j_D f \|_{B(L^2(X))}$, $j=1,2$, with $C$ depending only on some $S^0(\rr)$ seminorms of $g_\nu$.  Furthermore, since $g'_\nu(s)=O(|s|^{-1})$ as $|s|\to\infty$,  
$$
\|[D,f] g'_\nu(D) D\|\leq C \| {\rm ad}_D f\|_{B(L^2(X))}.
$$
In combination with \eqref{adD} we conclude the assertion for $\cI_\nu$. The proof for $\cI_\nu^\12$ is then analogous. 
\end{proof}

 \subsection{Energy estimates}  Suppose  $V \in \Diff^1_\b(X)$ is a real vector field with compact support. We denote $F=x^{-\nu+\12}$ and set 
  \beq\label{twistedV}
  V' := FVF^{-1} \in \Diffb^{1}(X), \quad \bar V' := \bar{F}V\bar F^{-1} \in \Diffb^{1}(X).
  \eeq
We will encounter expressions involving  the  commutator $[Q_i,V']$, which can be computed as follows: if $V = V^j \p_{z^j}$, then
\beq\label{eq:QV}
\bea {}
  [Q_i,V'] &= F[D_{z^i},F^{-1} V' F]F^{-1} \\ &= F[ D_{z^i}, V] F^{-1} 
  \\ & = \p_{z^i}(V^k)Q_k.
  \eea
\eeq
 Similarly, $[\bar Q_i,\bar V']=\p_{z^i}(V^k)\bar Q_k$.

   \medskip
   
   \noindent\textbf{Proof of  Theorem \ref{thm:Pinv}.}  The proof modifies the arguments in \cite{GW}, which in turns build on \cite{vasy2008propagation,warnick2013massive,vasy2012wave}, cf.~\cite{Dappiaggi2021} for a more complete and  general exposition which we follow  here.  
   
   We take $V=\phi W$ for a suitable function $\phi$, a real $\b$-vector field $W$ constructed as follows.
   
   By the global hyperbolicity assumption in Hypothesis \ref{hyp:global} and  there exists a Cauchy temporal function $t$ (in particular $\nabla_{\hat g} t$ is time-like) such that $W:=\nabla_{\hat g} t$ is tangent to $\p X$, see \cite{sanchez}.    For fixed $t_0<t_1$ we work on the interval $\open{t_0,t_1}$ and for small $\delta>0$ and $\varepsilon>0$ define the smooth function
   $$
  \chi(s)= \chi_0(-\delta(s-t_1))\chi_1((s-t_0)/\varepsilon), 
   $$  
   where $\chi_1$ monotonically interpolates between $0$ on $\opencl{-\infty,0}$ and $1$ on  $\clopen{1,\infty}$, and $\chi_0(s)=\theta(s)\exp(s^{-1})$. Then $\chi\geq 0$, $\chi'\leq 0$ and
   $$
   \chi \leq  - \delta (t_1-t_0)^2 \chi'.
   $$ 
  We then take $\phi:=\psi (\chi \circ t)$ where $\psi\in C_{\rm c}^\infty(X)$ equals $1$ on a neighborhood $U_1$ of $\pX$, and supported in a neighborhood $U$ which in the sequel will be taken  as small as needed, in particular small enough to be contained in a chart neighborhood with coordinates $(x,y)$ as in \sec{ss:AdS}.  Then $V=\phi W\in \Diff^1_\b(X)$ is timelike in a neighborhood of $U_1$.
      
    Let $u\in H^{1,1}_{\nu,\b,\loc}(X)$  be supported in $[t_0+\varepsilon, t_1]\times Y$. Without loss of generality we can assume that $u$ is supported in a small neighborhood of $U$; otherwise we can replace $u$ by $\varphi u$ for an appropriate test function equal $1$ on a smaller neighborhood of $U$ and use support properties of $V$ in formulas involving $Pu$ below to get rid of $\varphi$ whenever it matters.  
    
 Let $V',\bar V'\in \Diff_\b^1(X)$ be the two twisted versions of $V$ defined in \eqref{twistedV}.  We consider the following generalized commutator expression:
  \beq\label{gencom}
  \bea
\mathcal{\tilde E}_0(u, \bar V' u) +\mathcal{\tilde E}_0( V' u,   u)&= -\langle \hat g^{ij} Q_i   u , {\bar Q_j} \cI_{\nu}  \bar V ' u \rangle_{L^2}  -\langle \hat g^{ij} Q_i V' u , {\bar Q_j} \cI_{\nu} u \rangle_{L^2}
\eea 
 \eeq
 On the one hand,  by Cauchy--Schwarz and Young inequality applied twice,
 \beq \label{eq:cs}
 \bea
   \big|\mathcal{\tilde E}_0(u, \bar V' u) +\mathcal{\tilde E}_0(  V' u,   u)\big|&= \big| \langle \hat g^{ij} Q_i   u , {\bar Q_j} \cI_{\nu}  \bar V' u \rangle_{L^2}  + \langle \hat g^{ij} \bar Q_i  \cI_{\nu} V' u , {Q_j}  u \rangle_{L^2}  
 \\
   & \leq C\big( \|  \phi^{1/2} FWF^{-1}  P u \|^2_{H_\nu^{-1}}  + \| \phi^{1/2} u\|^2_{H_\nu^{1}}  + \| \phi^{1/2}  Pu \|_{L^2}^2 \fantom + \| \phi^{1/2} FWF^{-1}  u\|_{L^2}^2   
 \eea
 \eeq
 In the first step we have used the commutations relations \eqref{eq:idqq0} and also $[\cI_\nu,\hat g^{ij}]=0$, as follows from the assumption that $\hat g$ is a product metric. 
 On the other hand, we claim that we can rewrite \eqref{gencom} as a stress energy tensor term modulo smaller size terms that can be absorbed. To that end we   compute 
 \beq\label{gencom2}
 \bea {}
   \mathcal{\tilde E}_0(u, \bar V' u) +\mathcal{\tilde E}_0( V' u,   u)&= \langle \hat g^{ij} Q_i   u , {\bar Q_j} [\cI_{\nu} ,\bar V '] u \rangle_{L^2}  + \langle \hat g^{ij} Q_j u , [{\bar Q_i},{\bar V'}] \cI_{\nu} u \rangle_{L^2} \fantom + \langle \hat g^{ij} [ Q_j,V']u, {\bar Q_i} \cI_{\nu} u  \rangle_{L^2} +  \langle [\hat g^{ij} ,V']Q_j u, {\bar Q_i} \cI_{\nu} u \rangle_{L^2} \fantom + \langle (V'+(\bar V')^*) \hat g^{ij} Q_j u , {\bar Q_i} \cI_{\nu} u \rangle_{L^2}.
  \eea
\eeq
Observe that $[\hat g^{ij},V'] = [\hat g^{ij}, V] {=[\hat g^{ij},\bar {V'}]}$. Moreover, ${(\bar V')^*} = { F^{-1}} V^*  { F}$ and thus 
$$
\bea
 {(\bar V')^*} &= -{ F^{-1}}V { F} - {\rm div}_{\hat g} V + (n-2)x^{-1}V(x) \\
 &= -V + { F}V({ F^{-1}}) - {\rm div}_{\hat g}V + (n-2)x^{-1}V(x).
\eea
$$
 Note also $V'=FVF^{-1}=V+ F V(F^{-1})$. Recalling that $V=\phi W$, this serves to show that \eqref{gencom2} equals
 \begin{equation} \label{eq:energycommutator}
  \langle (2T^{ij} + K^{ij}) Q_i u, {\bar Q_j} \cI_{\nu} u\rangle_{L^2} +  \langle \hat g^{ij} Q_i   u , {\bar Q_j} [\cI_{\nu} ,\bar V '] u \rangle_{L^2} 
 \end{equation}
 where 
 the tensor $K$ is
 \begin{equation}
 K = - (\phi\cdot {\rm div}_{\widehat{g}}W + 2 F\phi V(F^{-1})+ (n-2)\phi x^{-1}W(x) ) \widehat{g}^{-1} 
 + \phi\mathcal{L}_W \widehat{g}^{-1}
 \end{equation}
 and $2T= 2 T_{\hat g}(W,\nabla_{\widehat{g}} \phi)$ is two times the stress-energy tensor with respect to $\hat g$ associated with $W$ and $\nabla_{\widehat{g}} \phi$, i.e.
 $$
 T_{\hat g}(W, \nabla_{\hat g} \phi) =  (\nabla_{\hat g}\phi) \otimes_s W - \tfrac{1}{2} \hat g(\nabla_{\hat g} \phi,W) \cdot  \hat g^{-1}.
 $$
 
 We can rewrite  \eqref{eq:energycommutator} as
 \beq\label{eq:ec2}
 \bea {}
&2\bra T^{ij}  \cI_{\nu(i)}^{1/2}Q_i u, \cI_{\nu(j)}^{1/2}Q_j u \ket_{L^2}  +   \langle  K^{ij} Q_i u, {\bar Q_j} \cI_{\nu} v\rangle_{L^2} \\ &  + 2 \langle [\cI_{\nu(i)}^{1/2},  T^{ij}  ]  Q_i u, \cI_{\nu(j)}^{1/2} {Q_j}  v\rangle_{L^2} +  \langle \hat g^{ij} Q_i   u , {\bar Q_j} [\cI_{\nu} ,\bar V '] u \rangle_{L^2} 
\eea
 \eeq
 where  $\nu(0)=\nu-1$ and $\nu(j)=\nu$ for $j=1,\dots,d$. Since $W$ is strictly time-like in the region of interest, the stress-energy tensor term involving $T^{ij}$ is negative definite and controls a homogeneous version of the $H_\nu^1$ norm of  $(-\phi')^{1/2} u $.  The arguments in \cite{GW} leading to a twisted version of the Poincaré inequality \cite[Prop.~2.5]{vasy2012wave} remain valid, so we can also control $L^2$ norms and deduce the estimate 
 $$
-\bra T^{ij}  \cI_{\nu(i)}^{1/2}Q_i u, \cI_{\nu(j)}^{1/2}Q_j u \ket_{L^2}  \geq      C   \| (-\phi')^{1/2} u \|_{H_\nu^1}^2
 $$
 for the actual $H_\nu^1$ norm.
 The commutator terms in  \eqref{eq:ec2} can be absorbed into the $H_\nu^1$ norm. Indeed, after  possibly shrinking $U$ and taking $\lambda$ small enough, 
 $$
 \| [\cI_{\nu(i)}^{1/2},  T^{ij}  ]\|_{B(L^2(U))} \leq \delta, \quad  \| [\cI_{\nu},  \bar V' ]\|_{B(L^2(U))} \leq \delta
 $$ 
 by Lemma \ref{lem:com}  and by writing  $\bar V'$ as a $\b$-differential operator and using that $[\cI_{\nu}, x\p_x]=0$. 
  Finally, the $K^{ij}$ term can be absorbed as well since it is estimated by $\| \phi^{1/2}u \|_{H_\nu^1}$ with $\phi^{1/2}$ small relative to  $(-\phi')^{1/2}$. Then by \eqref{eq:cs}  this is controlled by $\|    P u \|^2_{H_{\nu,\b}^{-1,1}}$ plus terms which are small relative to  $\| (-\phi')^{1/2}u \|_{H_\nu^1}$ and which consequently can be absorbed. In summary, 
 $$
\|(-\phi')^{1/2} u \|_{H_\nu^1}^2\leq C \|    P u \|^2_{H_{\nu,\b}^{-1,1}}.
 $$
  From that point on the abstract functional analytic arguments in \cite{vasy2012wave} can be repeated verbatim. 
  
 Finally, the microlocal propagation estimates leading to well-posedness in spaces with arbitrary conormal regularity in   \cite{GW,vasy2008propagation} can be easily adapted through  energy form modified with $\cI_{\nu}$ as before, the main point being that additional commutators of $\cI_{\nu}$ with elements of $\Psi^m_\b(X)$ are controlled  using Lemma \ref{lem:com}.
  \qed

 \medskip
 
 {\small
 \subsubsection*{Acknowledgments} The authors thank Yuchao Yi for helpful remarks. This work has received funding from the European Research Council (ERC) under the European Union's Horizon 2020 research and innovation program through the Consolidator Grant agreement~862342 (A.E.). A.E.'s research is also partially supported by the grants 
 CEX2023-001347-S and PID2022-136795\-NB-I00. The research of G.U.~is partly supported by NSF and a Robert R.~Phelps and Elaine F.~Phelps Professorship at University of Washington. M.W.~gratefully acknowledges support from the grant ANR-20-CE40-0018.}
 
 \bibliographystyle{abbrv}
 \bibliography{adsinverse}

 \end{document}